\PassOptionsToPackage{usenames,dvipsnames}{xcolor}
\documentclass[11pt,reqno,oneside]{amsart}
\usepackage{tikz}
\usetikzlibrary{positioning,chains,fit,shapes,calc}
\usepackage[dvipsnames]{xcolor}
\usepackage[citecolor=PineGreen,colorlinks=true,linkcolor=RoyalBlue]{hyperref}
\usepackage{enumitem,bm}
\usepackage[all,arc]{xy}
\usepackage{mathrsfs,graphicx,subcaption,sidecap,mathtools,setspace,xparse}
\usetikzlibrary{cd}

\usepackage[mathlines]{lineno}
\usepackage{etoolbox} 

\newcommand*\linenomathpatch[1]{%
  \expandafter\pretocmd\csname #1\endcsname {\linenomath}{}{}%
  \expandafter\pretocmd\csname #1*\endcsname{\linenomath}{}{}%
  \expandafter\apptocmd\csname end#1\endcsname {\endlinenomath}{}{}%
  \expandafter\apptocmd\csname end#1*\endcsname{\endlinenomath}{}{}%
}
\newcommand*\linenomathpatchAMS[1]{%
  \expandafter\pretocmd\csname #1\endcsname {\linenomathAMS}{}{}%
  \expandafter\pretocmd\csname #1*\endcsname{\linenomathAMS}{}{}%
  \expandafter\apptocmd\csname end#1\endcsname {\endlinenomath}{}{}%
  \expandafter\apptocmd\csname end#1*\endcsname{\endlinenomath}{}{}%
}
\expandafter\ifx\linenomath\linenomathWithnumbers
  \let\linenomathAMS\linenomathWithnumbers
  \patchcmd\linenomathAMS{\advance\postdisplaypenalty\linenopenalty}{}{}{}
\else
  \let\linenomathAMS\linenomathNonumbers
\fi

\linenomathpatch{equation*}
\linenomathpatchAMS{gather}
\linenomathpatchAMS{multline}
\linenomathpatchAMS{align}
\linenomathpatchAMS{alignat}
\linenomathpatchAMS{flalign}

\usepackage{dsfont}
\usepackage{kpfonts}
\usepackage[T1]{fontenc}

\usepackage{microtype} 
\DeclareFontFamily{OT1}{rsfs}{}
\DeclareFontShape{OT1}{rsfs}{n}{it}{<-> rsfs10}{}
\DeclareMathAlphabet{\mathscr}{OT1}{rsfs}{n}{it}

\theoremstyle{plain} 
\newtheorem{thm}{Theorem}[section]
\newtheorem{cor}[thm]{Corollary}
\newtheorem{prop}[thm]{Proposition}
\newtheorem{lem}[thm]{Lemma}

\theoremstyle{definition}
\newtheorem{defn}[thm]{Definition}

\newtheorem{exmp}[thm]{Example}

\theoremstyle{remark}
\newtheorem{rmk}[thm]{Remark}

\makeatletter
\def\paragraph{\@startsection{paragraph}{4}%
  \z@\z@{-\fontdimen2\font}%
  {\normalfont\it}}
\makeatother

\DeclareMathOperator{\e}{\mathcal{E}}

\renewcommand{\and}{\qquad \text{and} \qquad}
\newcommand{\Z}{\mathds{Z}}
\newcommand{\Cn}{\Z/N\Z}
\newcommand{\R}{\mathds{R}}

\newcommand{\W}{\mathcal{W}}
\newcommand{\s}{\mathbf{s}}

\newcommand{\ep}{\varepsilon}
\renewcommand{\phi}{\varphi}
\newcommand{\ba}{\alpha_\star}
\renewcommand{\sp}{\,\mathbf{a}}
\newcommand{\charf}{\mathds{1}}
\newcommand{\mus}{\mu_{S,\sp}}
\newcommand{\mual}{{\mu_{s,\alpha}}}

\renewcommand{\musl}{{\mu_0}}
\newcommand{\pua}{{p_{N,\alpha}}}

\DeclarePairedDelimiter{\oldnormaux}{\bracevert}{\bracevert}

\NewDocumentCommand{\cl}{som}{%
  \IfBooleanTF{#1}
    {\oldnormaux*{#3}}
    {\IfNoValueTF{#2}
       {\oldnormaux*{\vphantom{dq}#3}}
       {\oldnormaux[#2]{#3}}%
    }%
}
\newcommand{\dia}{D_{S, \sp}}
\newcommand{\cna}{c_{N,\alpha}}

\newcommand{\cost}[1]{ ||#1||_{S,\sp}}
\newcommand{\ncost}[1]{ {|||#1|||_{S, \sp}} }

\newcommand{\floor}[1]{\left\lfloor #1 \right\rfloor}
\newcommand{\ceil}[1]{\lceil #1 \rceil}
\renewcommand{\angle}[1]{{\langle #1 \rangle}}
\newcommand{\aI}{ {\alpha_1} }
\newcommand{\aII}{ {\alpha_2} }
\newcommand{\aIII}{ {\alpha_3} }
\newcommand{\ai}{{\alpha_i}}

\newcommand{\dell}{ {d(\ell,k)} }
\def\centerarc[#1](#2)(#3:#4:#5)
{ \draw[#1] ($(#2)+({#5*cos(#3)},{#5*sin(#3)})$) arc (#3:#4:#5); }

\DeclareMathOperator*{\argmin}{argmin}
\DeclareMathOperator{\var}{Var}

\numberwithin{equation}{section}

\title{Random walks on finite nilpotent groups driven by long-jump measures}
\author[L.~Saloff-Coste and Y.~Wang]{Laurent Saloff-Coste and Yuwen Wang}

\thanks{
  Y.W. is partially supported by NSF grant DMS-1645643 and Austrian Science Fund (FWF) project 34129.
  In addition, Y.W. and L.S-C. are both partially supported by NSF grant DMS-1707589.
}
\date{\today}

\begin{document}


\begin{abstract}
We consider a variant of simple random walk on a finite group.
At each step, we choose an element, $s$, from a set of generators (``directions'') uniformly, and an integer, $j$, from a power law distribution  (``speed'') associated with the chosen direction,
and move from the current position, $g$, to $gs^j$.
We show that if the finite group is nilpotent,
the time it takes this walk to reach its uniform equilibrium is of the same order of magnitude as the diameter of a suitable pseudo-metric on the group, which is attached the generators and speeds.
Additionally, we give sharp bounds on the $\ell^2$-distance between the distribution of the position of the walker and the stationary distribution, and compute the relevant diameter for some examples.
\end{abstract}

\maketitle


\section{Introduction}
\label{sect:intro}

A probability measure $\mu$ on a finite group $G$ with identity $e$ induces a Markov kernel
$$K(x,y) = K(e,x^{-1}y) = \mu(x^{-1}y), \; x,y \in G.$$
The associated iterated kernel $K^n$ is then given by
$$K^n(x,y) = \mu^{(n)}(x^{-1}y),$$
where $\mu^{(n)}$ is the iterated convolution of $\mu$ with itself $n$ times.
In this paper, we consider only probability measures $\mu$ that are symmetric, i.e. $\mu(g) = \mu(g^{-1})$, irreducible and aperiodic. Consequently, $K^n(x,y)$ converges to $\pi(y)$ as $n$ goes to infinity, where $\pi$ is the uniform distribution on $G$.
A \emph{random walk on $G$ driven by $\mu$}
is a sequence of $G$-valued random variables $\{X_n\}_{n\geq 0}$ of the form
\[ X_n = \xi_0 \xi_1 \dotsb \xi_n, \]
where $\xi_0$ is the initial, possibly random, position and $(\xi_i)_{i \geq 1}$ is an i.i.d. sequence of random variables with common distribution $\mu$.
When $\xi_0$ is $e$, the distribution $X_n$ is $K^n(e,\cdot)$.
The mixing time of such a random walk is
\[ t_{mix}  = \min \{ n: ||K^n(e, \cdot) - \pi||_{TV}  \leq 1/4 \}. \]

In a simple random walk with respect to a set $S \subseteq G$, at each time step, the walker chooses $s$ uniformly from $S$, and steps from her current location, $g$, to $gs$.
In this paper, we consider a variant of this walk, where the walker may ``jump'' further away in the direction of $s$, not just take a single step.
Specifically, in a finite group $G$, let $S = (s_1, \dotsb, s_k)$ be a $k$-tuple of elements that generate $G$.
For each $i$, let $\mu_i$ be a symmetric distribution supported on the cyclic subgroup $\langle s_i \rangle$ and set
\begin{equation}\mu (g) = \frac{1}{k} \sum_{i = 1}^k \mu_i(g). \label{eqn:intro-mu} \end{equation}

One natural choice is to set $\mu_i$ to be the uniform measure on $\angle{s_i}$, for each $i$.
In this case, the probability given to $s_i$ by $\mu_i$ varies drastically depending on the order of $s_i$.
In this paper, we consider a model that is more regular, which we call \emph{long-jump random walk}, and is inspired by classical stable processes, see \cite[Ch.\ 6]{Fel1971:xxiv+669},
and approximation algorithms of convex bodies \cite{KanLovSim1997:1}.
Before defining the $\mu_i$'s that we will use for the rest of the paper, we describe a more intuitive wrap-around model that is comparable.
We associate with each $i$ in $\{1, \cdots, k\}$ a number $\alpha_i \in (0,2)$
and the probability distribution $q_i(x) = c_{\alpha_i} (1+|x|)^{-(1+\alpha_i)}$, $x \in \Z$.
After choosing $i$ uniformly in $\{1, \cdots, k\}$, $j$ is chosen
from the probability distribution $q_i$,
and, in this time step, the walker jumps from the current location $g$ to $gs^j_i$.
Thus,
for a fixed $i$, the smaller the $\alpha_i$ the larger the probability that a high power of $s_i$ is chosen.
We will actually work with the following variant.

\begin{defn}
    Let $G$ be a finite nilpotent group of nilpotency class $\ell$, $S$ be a $k$-tuple of elements that generate $G$ and $\sp = (\alpha_1, \dotsc, \alpha_k) \in (0,2)^k$.
  A \emph{long-jump measure} on $G$ is
  \begin{equation}
    \mus (g) = \frac{1}{k} \sum_{i = 1}^k \sum_{j \in \Z/N_i \Z} \charf_{s_i^j}(g) p_i(j),
  \end{equation}
  where $N_i$ is order of $s_i$ in $G$ and $p_i: \Z/N_i \Z \to \R$ is
  \begin{equation}
    p_i (j) = \frac{c_i}{(1+\cl{j})^{1+\alpha_i}}, \qquad \text{where $0 \leq j < N_i$, $\cl{j} = \min(j, N_i-j)$, } c_i^{-1} = \sum_{j \in \Z/N_i\Z} \frac{1}{(1+\cl j)^{1+\alpha_i}}.
    \label{eq:pi-defn}
  \end{equation}
  In addition we will use
  \begin{equation}
    \mus (g) = \frac{1}{k} \sum_{i = 1}^k \mu_i (g), \qquad \text{where } \mu_i(g) = \sum_{j \in \Z/N_i \Z} \charf_{s_i^j}(g) p_i(j).
    \label{eq:mui-defn}
  \end{equation}
  An \emph{$(S, \sp)$-long-jump random walk} on $G$ is a random walk driven by a long-jump measure $\mus$.
  \label{defn:longjump}
\end{defn}

\noindent The relationship between this definition and wrap-around model is discussed in Appendix \ref{sect:mus-equivalence}.

Fix a $C_0 >0$. When $G$ is a finite nilpotent group of class $\ell < C_0$ and generated by a symmetric set $S$ of size less than $C_0$, it is shown in \cite{DiaSal1994:1} that the mixing time of the simple random walk associated with the set $S$ is of the same order as the square of the diameter of $G$ with respect to $S$.
In this paper, we answer the following question:
\begin{center}
  \emph{Given the modification of long-jumps, what is the mixing time?}
\end{center}
This a challenging question even in the case when $G = \Cn$ as illustrated by the following example.

\begin{exmp}
  For any positive integer $t$,
  consider the cyclic group $\Cn$ with $N = t^5$.
  Let $S$ be the generating $2$-tuple $(1,s)$, where $s = t^4$.
  The simple random walk driven by the uniform measure on $\{\pm 1, \pm s\}$ mixes in order $D^2$ where $D$ is the diameter of the Cayley graph $(\Cn,\{\pm 1, \pm s\})$.
  The diameter $D$ is of order $N^{4/5}$ and the mixing time of the simple random walk is of order $N^{8/5}$.

  Now, consider the random walk driven by $\mus$ where $\sp = (\alpha,1)$ and $\alpha \in (0,2)$.
  Intuitively, as $\alpha$ decreases, since the variance of the length of the steps increases, one expects that the mixing time decreases.
  Indeed, following from the results of this paper, the mixing time is of order
  \begin{equation}
    \left\{
    \begin{array}{ll}
      N^{\alpha} & \text{for } 0 < \alpha \leq 1/5, \\
      N^{1/5} & \text{for } 1/5 \leq \alpha \leq 1/4, \\
      N^{4\alpha/5} & \text{for } 1/4 \leq \alpha < 2.
    \end{array}
    \right
    .
  \label{eqn:diamex}
  \end{equation}
\label{ex:intro}
\end{exmp}

Let us summarize what is behind this mixing time estimate.
For a general finite nilpotent group, $G$, and a probability measure $\mus$, we introduce a quasinorm $\cost{\cdot}$ and show that the associated random walk mixes in time of order $\dia$, the diameter of $G$ with respect to $\cost{\cdot}$. Even when $G = \Cn$, as in Example \ref{ex:intro}, computing the diameter $\dia$ is a non-trivial task, see Appendix \ref{sect:algcycle}. The estimate (\ref{eqn:diamex}) is obtained using
this method.

We now explain how to associate with $(S, \sp)$ a quasi-norm on $G$.
Referring to the notation in Definition \ref{defn:longjump},
from $S = (s_1, s_2, \dotsc, s_k)$, create a formal alphabet $\def\s{\mathbf{s}} \mathcal{S} = \{\s_1^{\pm 1}, \s_2^{\pm 1}, \dotsc, \s_k^{\pm 1}\}$.
Let $\mathcal{W}$ be the set of finite words generated by $\mathcal{S}$ and $\def\deg{{\textrm{deg}}} \deg_{\s_i}(w)$ be the number of times either $\s_i^{+1}$ or $\s_i^{-1}$ appears in the word $w$.
There is natural projection map $\rho$ from $\mathcal{W} \to G$, mapping $\s_i^{\pm 1}$ to $s_i^
{\pm 1}$, $i = 1, \dotsc, k$.
For example, $w_1 = \s_1^{+1}$ and $w_2 = \s_1^{+1} \s_1^{-1} \s_1^{-1}$ both map to $s_1$ under $\rho$. However, their degrees with respect to $\s_1$ differs:
\begin{equation*}
\deg_{\s_1} (w_1) = 1 \qquad\qquad \deg_{\s_1} (w_2) = 3.
\end{equation*}
For $\sp \in (0,2)^k$, define the \emph{cost} of $g$ to be
\begin{equation}
  \cost{g} = \min_{w \in \mathcal{W}:g = \rho(w)} \left\{ \max_{i} \left\{ (\deg_{\s_i} (w))^{\alpha_{i}} \right\} \right\}.
  \label{eqn:cost}
\end{equation}
The function $\cost{\cdot}: G \to \R_+$ is a \emph{quasi-norm} on $G$: it only satisfies the triangle inequality up to a multiplicative constant of $2$ because we assume $\alpha_i \in (0,2)$.
We define $\dia$ to be the diameter of the quasi-norm $\cost{\cdot}$,
that is, the largest $\cost g$ can be when $g$ varies over $G$.

We are ready to state the main results of the article, which relate the spectral gap and
mixing time of $(S,\sp)$-long jump random walks to $\dia$.
Because the long-jump random walk is symmetric, its spectrum has the form
\[ -1 \leq \beta_{min} \leq \dotsb \leq \beta_1 < \beta_0 = 1.\]
The eigenvalue $\beta_{\min}$ is bounded away from $-1$ by a constant, see (\ref{eqn:betamin}) below. Concerning $\beta_1$ and mixing time, we prove the following theorems.

\begin{thm}
  Fix $C_0 > 0$ and  $0 < \ep < 1$. There exist $c_1, c_2 > 0$ (depending on $C_0$ and $\ep$) such that
  for any $\ell, k < C_0$, $\sp \in (\ep, 2-\ep)^k$, and any finite nilpotent group $G$ of nilpotency class $\ell$ generated by a $k$-tuple $S$,
  the $(S,\sp)$-long jump random walk satisfies
   \[ c_1 / \dia \leq 1 - \beta_1 \leq c_2/\dia. \]
   \label{thm:intro-gap}
\end{thm}

\begin{thm}
  Fix $C_0 > 0$ and  $0 < \ep < 1$. There exist $c_1, c_2 > 0$ (depending on $C_0$ and $\ep$) such that
  for any $\ell, k < C_0$, $\sp \in (\ep, 2-\ep)^k$, and any finite nilpotent group $G$ of nilpotency class $\ell$ generated by a $k$-tuple $S$,
  the $(S,\sp)$-long jump random walk satisfies
  \[ e^{- c_1 n / \dia} \leq || K^n(e, \cdot) - \pi||_{T.V.} \leq e^{-c_2 n /\dia}\]
  for all $n > 0$.
  In particular, $t_{mix}$ is of the same order as $\dia.$
  \label{thm:intro-mixing}
\end{thm}
\noindent Although the result stated here is with respect to the total variation norm, throughout the paper we will work with $\ell^2$, which will give more quantitative information; see Theorem \ref{thm:l2mixing}. \\

The main results above rely on volume growth properties of $\cost{\cdot}$, which we describe now. For a given quasi-norm $||\cdot||$, define
\begin{align*}
  B(x,r) = \{y \in G: ||x^{-1}y|| \leq r \} \quad \text{and} \quad
  V(x,r) = \sum_{y \in B(x,r)} \pi(y),
\end{align*}
where $\pi$ is the uniform measure on $G$.

\begin{defn}
    \label{defn:intro-doubling}
    A finite group $G$ equipped with a quasi-norm $||\cdot||$ is \emph{doubling} if there exists $A > 1$ so that
    \[V(2r) \leq A V (r), \qquad \text{for all $r \geq 0$.} \]
    We will call $A$ a \emph{doubling constant} for the pair $G$ and $||\cdot||$.
\end{defn}

\begin{thm}
Fix $C_0 > 0$ and  $0 < \ep < 1$. There exists $A> 0$ (depending on $C_0$ and $\ep$) such that
  for any $\ell, k < C_0$, $\sp \in (\ep, 2)^k$, and any finite nilpotent group $G$ of nilpotency class $\ell$ generated by a $k$-tuple $S$,
  the group $G$ equipped with $\cost{\cdot}$ is doubling with constant at most $A$.
\label{thm:intro-doubling}
\end{thm}

\paragraph{Organization} This paper is an extension of work done in \cite{SalZhe2015:1047, CheKum2018}; these papers are concerned with infinite groups, whereas this paper studies finite groups. Many of the techniques used in here take inspiration from proofs from those two papers; we will give specific citations as we use them.
In Section \ref{sect:gap}, we start by proving Theorem \ref{thm:intro-gap}. We prove the upper bound by using a pseudo-Poincar\'e inequality, where we rely heavily on results developed in \cite{SalZhe2015:1047}. For the lower bound, we use the Courant-Fischer characterization of $\beta_1$ with a test function and bounds that are similar to those in \cite{CheKum2018}.
In Section \ref{sect:vol}, we prove the doubling property, i.e., Theorem \ref{thm:intro-doubling} using growth results from \cite{SalZhe2015:1047} for free nilpotent groups and a lemma of \cite{Gui1973:379} to transport the result to finite nilpotent groups.
In Section \ref{sect:mixing}, we give precise mixing $\ell^2$-estimates, which gives a proof of Theorem \ref{thm:intro-mixing} as a corollary. For these results, we use now standard techniques of Nash inequalities developed in \cite{DiaSal1994:1,DiaSal1996:459}.
The $\ell^2$-mixing upper bound for time less than $\dia$ uses intermediate Nash inequalities.
The matching lower bound uses the spectral lower bound on balls from Section \ref{sect:gap} and relate it to $\mus^{(n)}$ by using an argument inspired by proofs from \cite{CouGri1997:133,SalZhe2016:4133}.

In Section \ref{sect:dia}, we discuss some of the challenges in computing $\dia$ in general by providing some illustrative examples.
Up until this point in the paper, we have assumed that $\sp \in (0,2)^k$. In Section \ref{sect:difalphas}, we explain how to generalize the main results when $\sp \in (0,\infty)^k$ by changing the definition of $\cost{\cdot}$

In Appendix \ref{sect:mus-equivalence}, we discuss the relationship between the intuitive model presented in paragraph 2 of the introduction and the model that we presented in Definition \ref{defn:longjump} with which we work throughout the paper. In Appendix \ref{sect:dirform}, we outline properties of the Dirichlet form and present explicit computations for the more laborious bounds used in Section \ref{sect:gap}. In Appendix \ref{sect:algcycle}, we present an algorithm for computing $\dia$ when $G = \Cn$, $S = (1,s)$ and $\sp = (\aI, \aII)$, and a proof of its correctness. We use this algorithm for many of the examples in Section \ref{sect:dia}.
\\

\paragraph{Notation}
\label{para:notation}
We conclude the introduction with some notation that we will use in the paper.

When $G$ is the cyclic group $\Cn$, and $g \in \Cn$ is represented as a number in $[0,N-1]$, it will be convenient to define $\left \bracevert g \right \bracevert = \min( |g|, |N-g|)$.

All Markov kernels $K:G \times G \to \R$ considered in this paper are symmetric and irreducible, and their stationary distributions $\pi$ are uniform on $G$.
Note that for random walks on groups driven by $\mu$, $K(x,y) = K(e, x^{-1}y) = \mu(x^{-1}y)$.
Define $Kf(x) = \sum_{y \in G} K(x,y) f(y)$.
The corresponding continuous-time Markov chain has kernel $H_t = e^{-t(I-K)} = e^{-t}  \sum_{n=0}^\infty \frac{t^n}{n!} K^n$.
Let $k^n_e (x) = K^n(e,x)/\pi(x)$ and $h_t^e (x) = H_t(e,x)/\pi(x)$ be the densities with respect to $\pi$ of the discrete- and continuous-time kernels.

The space $\ell^p(\pi)$ is the set of functions from $G$ to $\R$ under the norm
$$||f||_p = \left(\sum_{x \in G} |f(x)|^p \pi(x) \right)^{1/p}$$ if $p \geq 1$ and
$||f||_\infty = \sup_{x\in G} |f(x)|.$
Given $p,q \in [1,\infty]$ and $K: \ell^p(\pi) \to \ell^q(\pi)$, define
$$||K||_{p \to q} = \sup_{f \in \ell^p(\pi)} \left \{ \frac {||Kf||_q}{||f||_p} \right \}.$$
The inner product on $\ell^2(\pi)$ we will use is $\angle{f,g}_\pi = \sum_x f(x) g(x) \pi(x)$.
The Dirichlet form associated with $\mu$ on $\ell^2(\pi)$ is
$$\e_\mu(f,g) = \angle{(I-K)f,g} = \frac12 \sum_{x,y} (f(x) - f(xy))(g(x)-g(xy)) \mu(y) \pi(x).$$
The relation $f \asymp g$, where $f$ and $g$ are positive functions,
means that there exist constants $c_1, c_2 > 0$ so that $c_1 f \leq g \leq c_2 f$.

\section{Spectral gap estimates}
\label{sect:gap}

The main tool we use to study the spectral gap is the Dirichlet form.
It is related to the spectral gap by
\begin{equation}
  1-\beta_1 = \min_{\substack{\angle{f,\textbf{1}}=0 \\ f \neq 0}}
  \left\{ \frac{\e_\mu(f,f)}{||f||_2^2} \right\} = \min_{f \neq 0 } \left\{ \frac{\e_\mu(f,f)}{\var_\pi(f)} \right\},
  \label{eqn:gap}
\end{equation}
which is explained in \cite[Section 2]{Saloff-Coste1997}.
Moreover, the form is linear in $\mu$, so bounds for $\e_{p_i}$ can be aggregated to a bound for $\e_{\mus}$. The details of these computations are included in Appendix \ref{sect:dirform}.

Define $\alpha_* =  \min_{\alpha \in \sp} \frac{\alpha}{2(1+\alpha)}$. It follows from Lemma \ref{lem:cna_indep} that
  $$\alpha_*
    \leq \min_{1 \leq i \leq k} c_i
    \leq \min_{1 \leq i \leq k} p_i (e) \leq \mus(e).$$
By a standard bound, see e.g. Theorem 6.6 of \cite{Sal2004:263},
\begin{equation}
  \label{eqn:betamin}
  \beta_{min} \geq 2\mus(e)-1 \geq 2\alpha_*-1.
\end{equation}

\subsection{Spectral gap lower bound}
A probability distribution on $G$, $\mu$, satisfies \emph{the pseudo-Poincar\'e inequality} if,
for any $r > 0$,
there exists
$a(r) > 0$ such that, for all $f \in \ell^2(\pi)$,
\[ ||f-f_r||_2^2 \leq a(r) \e_\mu(f,f), \]
where
\begin{equation}
f_r (x) = \frac{1}{V(x,r)} \sum_{y \in B(x,r)} f(y) \pi(y)
\and V(x,r) = \sum_{y \in B(x,r) } \pi (y).
\label{eqn:defnfr}
\end{equation}
We need the following result regarding nilpotent groups.
\begin{thm}\cite[Theorem 2.10] {SalZhe2015:1047}
  Let $k$ and $\ell$ be positive integers and $\sp \in (0,2)^k$. There exist $C = C(\ell, k, \sp)$, $p = p(\ell, k, \sp)$, and $(i_1, \dotsc, i_p) \in \{1, \dots, k\}^p$ so that for any finite nilpotent group $G$ of class $\ell$ and generating $k$-tuple $S$,
any $g \in G$ with $\cost{g} \leq r$ can be written as
\begin{equation*}
  g = \prod_{j=1}^p s_{i_j}^{m_j} \quad \text{with } |m_j| \leq Cr^{1/\alpha_{i_j}}.
\end{equation*}
\label{thm:rewritingwords}
\end{thm}

\begin{proof}
  Let $\hat G = N(\ell, k)$ be the free nilpotent group of nilpotency class $\ell$ and generated by $S$. Theorem 2.10 of \cite{SalZhe2015:1047} states that there exist an integer $p = p(\ell,k,\sp)$, a constant $C = C(\ell,k,\sp)$, and $(i_1, \dotsc, i_p) \in [k]^p$, such that for all $\hat g \in \hat G$ that can be expressed a word $\hat w$ where $\deg_{\s_i} \hat w \leq r^{1/\alpha_i}$, $\hat g$ can be rewritten as
  \begin{equation}\label{eqn:rewritefree} \hat g = \prod_{j=1}^p s_{i_j}^{m_j} \quad \text{with } |m_j| \leq C r^{\alpha_{i_j}}. \end{equation}

  Define $\rho$ and $\hat \rho$ be the projections maps from $\W$ to $G$ and $\hat G$ respectively, mapping $\s_i \to s_i$.
  There exists a group homomorphism $\phi$ so that the following diagram commutes, i.e., such that $\phi(s_i) = s_i$ for all $i$.
  \begin{center}
      \begin{tikzcd}
          & & \hat G = N(\ell,k) \arrow[d,"\phi"]\\
          S \arrow[hook, r, "i"] & \W \arrow[ru,"\hat \rho"] \arrow[r,"\rho"] & G
      \end{tikzcd}
  \end{center}
  Let $g \in G$ be an element satisfying the conditions in the theorem. Let $w_0$ be a word that realizes $\cost{g}$ in the sense of (\ref{eqn:cost}), and $\hat g = \hat \rho(w_0)$. Since for all $i$, $\deg_{\s_i}  w_0 \leq r^{1/\alpha_{i}}$, there exist $p(\ell, k ,\sp)$, $C(\ell,k,\sp)$ and $(i_1, \dotsc, i_p)$ so that (\ref{eqn:rewritefree}) is satisfied. After applying $\phi$ to both sides, we get the desired result.
\end{proof}

\begin{thm} \cite[Theorem 4.3]{SalZhe2015:1047}
  Let $k$, $\ell$, and $\sp \in (0, 2)^k$ be fixed. There exists a constant $a = a(k, \ell, \sp)$ such that for any long-jump random measure $\mus$ on a finite nilpotent group $G$ of class $\ell$ with $|S|=k$ and $f: G \to \R$,
  \[||f-f_r||_2^2 \leq a r \e_{\mus} (f,f). \]
  \label{thm:nilpoincare}
\end{thm}

\begin{proof}
  Fix $r > 0$,  $y_0 \in B(e,r)$ and $w_0 \in \mathcal{W}$ so that $w_0$ realizes $\cost{y_0}$. By Theorem \ref{thm:rewritingwords}, there exists $C_0 = C_0(\ell, k, \sp)$, $p = p(\ell, k , \sp)$ and $(i_1, \dotsc, i_p ) \in [k]^p$ so that $y_0$ can be written as $y_0 = s_{i_1}^{m_1} \dotsb s_{i_p}^{m_p}$ where $|m_j| \leq C_0 r^{1/\alpha_{i_j}}$.
  For each $j$, Theorem \ref{thm:embedded} gives the existence of a constant $C_1(\alpha_{i_j}) > 0$ so that, for all $f: G \to \R$ and $m \in \Z/N_{i_j} \Z$ where $|m|^{\alpha_{i_j}} \leq C_0^{\alpha_{i_j} } r$,
   \begin{equation}
     \frac 1 {|G|} \sum_{x \in G} |f(x)-f(xs^{m})|^2 \leq C_1(\alpha_{i_j}) C_0^{\alpha_{i_j}} r \e_{\mu_{i_j}} (f,f).
     \label{eqn:orbitp}
   \end{equation}
  By Theorem \ref{prop:dirprops} (2), for all $f \in \ell^2(\pi)$,
  \begin{align*}
   \frac{1}{|G|} \sum_{x\in G} |f(x) - f(xy_0)|^2 &= \frac{1}{|G|} \sum_{x \in G} |f(x) - f(x s_{i_1}^{m_1} \dotsb s_{i_p}^{m_p})|^2\\
   &\leq \frac{p}{|G|} \sum_{j=1}^p \sum_{x \in G} |f(x) - f(x s_{i_j}^{m_j})|^2  \qquad \text{(telescoping sum and Cauchy-Schwarz)}\\
   &\leq C_1(\alpha_{i_j}) p \sum_{j = 1}^p C_0^{\alpha_{i_j}} r \e_{\mu_{i_j}} (f,f) \qquad \text{(by (\ref{eqn:orbitp}))}\\
   &\leq p^2 k \max_{1 \leq i \leq k} \left \{ C_1(\alpha_{i}) C_0^{\alpha_i} \right\} r \e_{\mus}(f,f) \qquad \text{(since $\e_{\mu_i} \leq k \e_{\mus}$)}  .
  \end{align*}
  Let $a = k p^2 \max_i \{C_1(\alpha_{i}) C_0^{\alpha_i} \}$.
  The theorem now follows immediately from Proposition \ref{prop:dirprops} (1).
\end{proof}

\begin{proof}[Proof of lower bound of Theorem \ref{thm:intro-gap}]
  Let $r = \dia$. In this case,
  $$f_r = \mathds{E}_\pi [f] \and ||f-f_r||_2^2 = \var_\pi(f).$$
  Using the setting of Theorem \ref{thm:nilpoincare},
  we have that for all $f \in \ell^2(\pi)$, $\var_\pi(f) \leq a \dia \e_{\mus}(f,f)$.
  From the spectral gap characterization (\ref{eqn:gap}), we obtain $1 - \beta_1 \geq 1/a \dia$.
\end{proof}

\subsection{Spectral gap upper bound}

In this section, we will prove the upper bound of Theorem \ref{thm:intro-gap} as a consequence of the following result, which is similar to \cite[Lemma 4.2]{CheKum2018}.

\begin{thm}
  Let $\mus$ be a long-jump measure on a finite group $G$ that is doubling with constant $A$ with respect to $\cost{\cdot}$.
  There exists $\zeta \in \ell^2 (\pi)$ and $a(A, \sp) > 0$ such that
  \[
    \frac{\e_{\mus} (\zeta, \zeta)}{||\zeta||_2^2} \leq \frac{a(A,\sp)}{\dia}.
  \]
  \label{thm:emuslower}
\end{thm}

\begin{proof}[Proof of the upper bound of Theorem \ref{thm:intro-gap}]
  Let $G$ be a finite nilpotent group with nilpotency class $\ell$, $S$ be a list of $k$ generating elements, and $\sp \in (0,2)^k$.
  This gives a long-jump measure $\mus$ and cost function $\cost{\cdot}$.
  By Theorem \ref{thm:intro-doubling}, $G$ is doubling with respect to $\cost{\cdot}$, with doubling constant $A(\ell,k,\sp)$.
  From Theorem \ref{thm:emuslower}, there exists a constant $a(\ell,k,\sp)>0$ and function $\zeta$ so that
  \[ \frac{\e_{\mus}(\zeta,\zeta)}{||\zeta||_2^2} \leq \frac{a(\ell,k,\sp)}{\dia}. \]
  From the spectral gap characterization (\ref{eqn:gap}), $1 - \beta_1 \leq a(\ell,k,\sp) / \dia$.
\end{proof}

We are left with the task of proving Theorem \ref{thm:emuslower}.

\begin{lem}
  Let $o \in G$ where $\cost{o} = \dia$.
  In addition, define
  \begin{align*}
    \Omega_+ = \left\{ x \in G: \cost x \leq \frac{1}{12} \dia \right\} \and
    \Omega_- = \left\{ x \in G: \cost {o^{-1} x} \leq \frac{1}{12} \dia \right\}.
  \end{align*}

  If $g \in \Omega_+$ and $gh \in \Omega_-$, then $\cost{h} \geq \frac{1}{8} \dia$. Thus $\Omega_+$ and $\Omega_-$ are disjoint.
  \label{lem:gamma}
\end{lem}

\begin{proof}
  We know that
  \begin{align*}
    \cost{o} \leq 2 ( \cost{o^{-1}g} + \cost {g} ) \and
    \cost{o^{-1}g} \leq 2 ( \cost{o^{-1}gh} + \cost{h} ) .
  \end{align*}
  It follows that
  \[ \dia = \cost{o} \leq 2 \left( 2 \left( \frac{1}{12} \dia + \cost{h} \right) + \frac{1}{12} \dia \right) = \frac 12 \dia + 4 \cost{h}. \]
  Thus, $\cost{h} \geq \dia / 8$.
\end{proof}

We now define the test function for Theorem \ref{thm:emuslower}.
For $R = \dia/16$, let $\zeta: G \to \R$ be
 $$\zeta (g) = \zeta_+ (g) - \zeta_-(g),$$ where
 $\ba = \min (\sp),$
  \begin{align}\label{eqn:defzetaplus}
    \zeta_+ (g) &= (R^{1/\ba} - \cost{g}^{1/\ba})_+ \and \zeta_- (g) = (R^{1/\ba} - \cost{o^{-1}g}^{1/\ba})_+.
  \end{align}
Because $R = \dia/16$, by Lemma \ref{lem:gamma}, the supports of $\zeta_+$ and $\zeta_-$ are disjoint and
$$||\zeta||_2^2 = ||\zeta_+||^2 + ||\zeta_-||^2
= 2 ||\zeta_+||^2 = 2 ||\zeta_+||^2.$$

  Let $A$ be the doubling constant of $G$ with respect to $\cost{\cdot}$ and $B(e,R)$ is a ball with respect to the quasi-metric $\cost{\cdot}$.
  Because $ R^{1/\ba} - \cost{g}^{1/\ba} \geq ( 1 - {2^{-1/\ba}}) R^{2/\ba}$ when $g \in B(e, R/2)$, it follows that
  \begin{align*}
    ||\zeta_+||_2^2 &= \frac{1}{|G|} \sum_{g \in G} ( R^{1/\ba} - \cost{g}^{1/\ba} )_+^2
    \geq \frac{1}{|G|} \sum_{g \in B(e, R/2)} \left( 1 - \frac{1}{2^{1/\ba}} \right)^2 R^{2/\ba} \\
    &= \frac{1}{|G|} \left( 1 - \frac{1}{2^{1/\ba}} \right)^2 R^{2/\ba} \# B(e,R/2)
    \geq \frac{1}{|G|} \left( 1 - \frac{1}{2^{1/\ba}} \right)^2 \frac{1}{A} R^{2/\ba} \# B(e,R).
  \end{align*}
  Thus,
  \begin{equation}
    ||\zeta||_2^2 \geq C_0 R^{2/\ba} \frac{\# B(e,R)}{|G|}, \qquad \text{where $C_0 = \frac 2{A} (1 - 2^{-1/\ba})^2$}.
    \label{eqn:l2testfct}
  \end{equation}

  Now Theorem \ref{thm:emuslower} follows from (\ref{eqn:l2testfct}) and the following lemma.
  \begin{lem}
    Let $\zeta$ be defined as above. Then there exists $C(k,\ell,\sp) > 0$ such that
  $$ \e_{\mus} (\zeta,\zeta) \leq \frac{C(k,\ell,\sp) \#B(e,R)}{|G|} R^{-1 + 2/\ba}.$$
  \label{lem:test1}
  \end{lem}

  Because
  $\e_{\mus} (\zeta, \zeta) = 2 \e_{\mus} (\zeta_+, \zeta_+) - 2 \e_{\mus} (\zeta_+, \zeta_-)$,
  Lemma \ref{lem:test1} reduces to the following statement.

  \begin{lem}
    Let $\zeta$ be defined as above. Then there exists $C(k,\ell,\sp) > 0$ such that
    \begin{equation} \label{eqn:zetapp}
      \e_{\mus} (\zeta_+,\zeta_+) \leq \frac{C(k,\ell,\sp) \#B(e,R)}{|G|} R^{-1 + 2/\ba}
    \end{equation}
    and
    \begin{equation} \label{eqn:zetapm}
      - \e_{\mus} (\zeta_+,\zeta_-) \leq \frac{C(k,\ell,\sp) \#B(e,R)}{|G|} R^{-1 + 2/\ba}.
    \end{equation}
  \label{lem:test2}
  \end{lem}

  \begin{proof}[Proof of (\ref{eqn:zetapp})]
    For this bound, we will use the notation from Definition \ref{defn:longjump} to describe $\mus$.
  Fix $i_0 \in [1, k]$, and let $s_0 = S(i_0)$, $\alpha_0 = \sp(i_0)$, $\mu_0 = \mu_{i_0}$ and $p_0 = p_{i_0}$.
  We will first prove the inequality from the theorem for each $i_0$ and then take the average of both sides for the final result.

  Keeping this notation in mind, we begin by giving an upper bound for
  \[ \e_{\mu_0} (\zeta_+, \zeta_+)  = \frac 1{2|G|} \sum_{g, h \in G} |\zeta_+ (gh) - \zeta_+ (g)|^2 \musl (h). \]
  Let $\Omega = \{ (g,h) \in G \times \langle s_0 \rangle: \zeta_+(gh) + \zeta_+(g) > 0 \}.$
  So we can restrict the sum above to just $\Omega$. For a fixed $h$, we have that
  \[
    \# \{ g \in G: (g,h) \in \Omega\} \leq 2 \# B(e, R).
  \]

  Note that $\musl(h)$ is only non-zero when $h \in \langle s_0 \rangle$, so we can write $h = s_0^t$.
  Thus, we can further break the sum into two parts: (1) when $ |t| > \rho$ and (2) when $|t| \leq \rho$, where $\rho = (12 R)^{1/\alpha_0}$.
  For the first sum, by Lemma \ref{lem:gamma}, we have
  \begin{align*}
     \sum_{(g,h):\Omega: |t| \geq \rho} |\zeta_+(gh) - \zeta_+(g)|^2 \musl(h) &\leq 2 (R^{1/\ba})^2 \# B(e, R) \sum_{|t| \geq \rho} p_0(t) \\
     &\leq \frac{4 R^{2/\ba} \#B(e, R)}{\alpha_0 \rho^{\alpha_0}}.
  \end{align*}

  For sum (2), fix $g \in G$ and $h \in \angle{s_0}$, and choose the smallest $t$ in absolute value so that $s_0^t=h$. It will be convenient, for all $g \in G$, to set $w_g$ to be a word that realizes the cost of $g$, and set $x = \deg_{\mathbf{s_0}} w_g$, the number of times either $\mathbf{s_0}$ or $\mathbf{s_0^{-1}}$ appears in $w_g$. To start, we would like to bound the term $|\zeta_+(gh) - \zeta_+(g)|$. Note that we can assume that $\cost{gh} \geq \cost{g}$; otherwise we can set $g_0 = gh$ and $g_0 h^{-1} = g$, and the bound would proceed the same since we do not make assumptions about $g$ and $\cost{h} = \cost{h^{-1}}$.
  This implies that
  \begin{align*}
    |\zeta_+(gh) - \zeta_+(g)| &= (R^{1/\ba} - \cost{g}^{1/\ba})_+ - (R^{1/\ba} - \cost{gh}^{1/\ba})_+.
  \end{align*}
  We will now show that this expression is less than or equal to $\cost{gh}^{1/\ba} - \cost{g}^{1/\ba}$.
  Since we assume that $\cost{gh} \geq \cost{g}$, if the first term is zero, then so is the second term. Therefore, three cases remain: (1) if both term are zero, then the inequality holds trivially, (2) if both terms are non-zero, then the two lines are equal, and (3) if the second term is zero, but the first is not, then, $\cost{gh} \geq R$, and
  \[
    |\zeta_+(gh) - \zeta_+(g)| = R^{1/\ba} - \cost{g}^{1/\ba} \leq \cost{gh}^{1/\ba} - \cost{g}^{1/\ba}.
  \]
  We are ready to evaluate
  \begin{align*}
    |\zeta_+(gh) - \zeta_+(g)| &\leq \max_{1 \leq i \leq k} \{ (\deg_{\s_i}w_g)^{\ai/\ba}, (x + |t|)^{\alpha_0/\ba} \} - \max_{1 \leq i \leq k} \{ (\deg_{\s_i} w_g)^{\ai/\ba}, x^{\alpha_0/\ba} \} \\
  &\leq (x + |t|)^{\alpha_0/\ba} - x^{\alpha_0/\ba}.
  \end{align*}
  By the fundamental theorem of calculus and $x \leq (12 R)^{1/\alpha_0}$, we have
  \[
    |\zeta_+(gh) - \zeta_+(g)| \leq \int_x^{x+|t|} \frac{\alpha_0}{\ba} s^{\frac {\alpha_0} \ba-1} \, ds
    \leq \frac{\alpha_0}{\ba} \left( \left(12 R \right) ^{1/\alpha_0} + \rho \right)^{\frac{\alpha_0}\ba-1} |t|
    \leq \frac{\alpha_0}\ba ( 2^{1/\alpha_0} 12 R)^{\frac1\ba - \frac {1} {\alpha_0}} |t|.
  \]
  Summing over $h = s^t$, where $|t| \leq \rho$,
  \begin{align*}
    \frac 12 \sum_{|t| \leq \rho} |\zeta_+(g s^t) - \zeta_+(g)|^2 \musl(t) &\leq \frac{\alpha_0^2}{\ba^2} \left( 2^{1/\alpha_0} 12  R \right)^{2/\ba-2/\alpha_0} \#B(e,R) \sum_{|t| \leq \rho} |t|^2 \musl(t) \\
    &\leq C_1 R^{2/\ba-2/\alpha_0} \#B(e,R) \rho^{2-\alpha_0} \qquad
    \text{(by Lemma \ref{lem:upperpi})},
  \end{align*}
  where $C_1 = C_1 (\alpha_0)= C_0(\alpha_0) \alpha_0^2 12^{2/\ba-2/\alpha_0} \frac{3^{2-\alpha_0}}{2-\alpha_0}$ and $C_0(\alpha_0)$ is the constant that appears in Lemma \ref{lem:upperpi}.
  Putting the two parts together, we have
  \begin{align*}
    |G| \e_{\mu_0} (\zeta_+,\zeta_+) & \leq \# B(e,R) R^{2/\ba}  \left( \frac{4}{\alpha_0}  \rho^{-\alpha_0} +  C_1 R^{-2/\alpha_0} \rho^{2-\alpha_0} \right) \\
   & \leq \# B(e,R) R^{2/\ba} \left(\frac{4}{\alpha_0 12^{\alpha_0}} R^{-1} + C_1 12^{2 - \alpha_0} R^{-2/\alpha_0} R^{2/ \alpha_0 - 1}  \right) \\
   & \leq C_2 \# B(e,R) R^{-1 + 2/\ba},
  \end{align*}
  where $C_2 = C_2 (\alpha_0) = \frac{4}{\alpha_0 12^{\alpha_0}} + C_1(\alpha_0) 12^{2 - \alpha_0}$.
  \end{proof}

  \begin{proof}[Proof of (\ref{eqn:zetapm})]
    We use the notation $s_0$, $\alpha_0$, $\mu_0$, and $p_0$ as above, and give a lower bound for
  \[
     \e_{\mu_0}(\zeta_+, \zeta_-) =  \frac 1 {2|G|}\sum_{g,h \in G} (\zeta_+(gh) - \zeta_+(g)) (\zeta_-(gh) - \zeta_-(g)) \musl(h).
  \]
  Let $\Omega_+$ be the support of $\zeta_+$ and $\Omega_-$ be the support of $\zeta_-$. As we chose $R = \dia /12$, Lemma \ref{lem:gamma} implies that $\Omega_+$ and $\Omega_-$ are disjoint. We see that the only non-zero summands are those where $g \in \Omega_+$ and $gh \in \Omega_0$ or vice versa, in which case $\cost{h} > R$.
  By symmetry, we have
  \begin{align*}
    - \e_{\mu_0}(\zeta_+, \zeta_-) &=  \frac{1}{|G|} \sum_{\substack{g \in \Omega_+ \\ gh \in \Omega     _-}} \zeta_+(g) \zeta_-(gh) \musl(h)
  \end{align*}
  Because $|\zeta_+|, |\zeta_+| \leq R^{1/\ba}$ and $\zeta_+$ has support in $B(e,R)$,
  \begin{align*}
    - |G| \e_{\mu_0}(\zeta_+, \zeta_-) &
    \leq R^{2/\ba} \sum_{\substack{g \in \Omega_+ \\ gh \in \Omega_-}} \musl(h)
    \leq R^{2/\ba} \# B (e, R) \sum_{\cost{h} > R} \musl(h) \\
    &\leq R^{2/\ba} \# B (e, R) \sum_{|t|^{\alpha_0} > R} p_0(t)
    \leq C_3(\alpha_0) \# B (e, R) R^{-1 + 2/\ba},
  \end{align*}
  where $C_3(\alpha_0)$ is as in Lemma \ref{lem:upperpi}.
  Averaging over all components $\mu_i$ of $\mus$, we get that the inequality also holds for $\mus$.
  \end{proof}

\section{Volume estimates}
\label{sect:vol}

For proving the doubling statement of Theorem \ref{thm:intro-doubling}, there are two main ingredients: (1) \cite[Example 1.5]{SalZhe2015:1047} that shows doubling with respect to $\cost{\cdot}$ for free nilpotent groups
and (2) the finite version of \cite[Lemma 1.1]{Gui1973:379} stated below, which allows us to translate doubling from the free nilpotent group to the finite nilpotent group.

\begin{lem} [{\cite[Lemma 1.1]{Gui1973:379}}]
  Let $G$ be a finitely-generated countable group acting on a set $X$, which we will write on the right.
  Let $A$ and $B$ be finite subsets of $G$, and $Y$ a subset of $X$. Then,
  \[ \# A  \#(Y B) \leq \#(AB) \#(Y A^{-1}) . \]
\label{lem:volproj}
\end{lem}

\begin{proof}[Proof of Theorem \ref{thm:intro-doubling}]
  Let $\hat G = N(\ell, k)$ denote the free nilpotent group of class $\ell$ generated by $S$ of class $\ell$.
  Let $\W$ be the set words generated by entries of $S$ and $\hat \rho$ and $\rho$ be the natural projection maps from $\hat G$ and $G$, respectively, to $\W$.
  Using this lifting, we can define an $(S, \sp)$-cost function on $\hat G$, which we will call $\ncost{\cdot}$ to differentiate.
  We will also use $\hat B$ and $B$ to denote balls with respect to $\ncost{\cdot}$ and $\cost{\cdot}$ respectively.
  Further, there exists $\phi: \W \to \hat G$ so that the following diagram commutes.
  \begin{center}
    \begin{tikzcd}
      & & \hat G = N(\ell,k) \arrow[d,"\phi"]\\
      S \arrow[hook, r, "i"] & \W \arrow[ru,"\hat \rho"] \arrow[r,"\rho"] & G
    \end{tikzcd}
  \end{center}
  With this, by the way that the cost function is defined, for all $\hat g \in N(\ell,k)$, $\ncost{\hat g} \geq \cost{\phi(\hat g)}$.

  By Example 1.5 from \cite{SalZhe2015:1047}, there exist constants $c_1, c_2 > 0$
  \begin{equation}
      c_1 r^\dell \leq \# \hat B (e,r) \leq c_2 r^\dell \quad \text{where } \dell = \sum_{m=1}^\ell \sum_{d|m} \boldsymbol{\mu}(d) k^{m/d}
  \end{equation}
  and $\boldsymbol \mu$ is the classical M\"obius function. Thus, $N(\ell,k)$ has polynomial growth with respect to $\ncost{\cdot}$.

  Next, we apply Lemma \ref{lem:volproj} with the group action of $\hat g \in \hat G$ on $x \in G$ via $ x \cdot \hat g = x \phi(\hat g)$, and the sets
  \begin{align*}
    Y &= \{ e_G\}, \\
    A &= \hat B (e, r) = \{ x \in N(\ell,k) : \ncost{x} \leq r\} \text{ and}\\
    B &= \hat B (e, 2r) = \{ x \in N(\ell,k) : \ncost{x} \leq 2r\}.
  \end{align*}
  First, notice that since $\cost{x}=\cost{x^{-1}}$, $YA^{-1} = YA$.
  We then show $YA^{-1} = YA = B(e, r)$ by showing inclusion both ways.
  Let $x \in B(e,r)$ and $w \in \W$ that realizes the cost of $x$, i.e., $\cost{x} = \max_i (\deg_{\s_i} w)^{\ai}$.
  Consider $\hat x = \hat \rho(w) \in N(k,\ell)$.
  By construction, $\varphi ( \hat x) = x$, $\ncost{\hat x} \leq r$, and thus $\hat x \in \hat B (e_G,r)$.
  It follows that $Y A \ni e_G \cdot \hat x =  e_G \phi(\hat x) = x,$
  and $B(e_G, r) \subseteq AY$.

  Now let $y \cdot a \in YA$. Let $w$ be a word that realizes cost of $a$, i.e., $\ncost{a} = \max_i (\deg_{\s_i} w)^{\ai}$.
  As in the previous case $\rho(w)$, which is also equal to $\phi(a) = y \cdot a$, must have cost less than or equal to $r$. Therefore, $YA \subseteq B (e_G,r)$.

  Next, we want to show that $AB \subseteq \hat B (e, 6r)$.
  Let $a \in A$ and $b \in B$.
  Let $w_0$, $w_1$, and $w_2$ be words that realizes the costs of $ab$, $a$ and $b$ respectively.
  By the triangle inequality for our quasi-norm, $\cost{ab} \leq 2 (\cost a + \cost b)$.
  So $A \subseteq \hat B (e, 6r)$.

  Finally, we can show doubling
  \begin{align*}
      \frac{\#B(e, 2r)}{\#B(e,r)} &= \frac{\# Y B }{\# Y A^{-1}}
                                      \leq \frac{\# AB}{\# A } \qquad \text{(by Lemma \ref{lem:volproj})} \\
                                      &\leq \frac{\# \hat B (e, 6r)}{\# \hat B (e,r)}
                                      \leq \frac{c_2 (6r)^\dell}{ c_1 r^\dell}
                                      \leq 6^\dell (c_2 / c_1). \qedhere
  \end{align*}
\end{proof}

Doubling imples the following property, which we use in Section \ref{sect:mixing}.
\begin{cor}
  For all $0 \leq r \leq R \leq \dia$,
  \[ V(e, r) \geq  A^{-2} V(e, R) \left( \frac{r+1}{R+1} \right)^d, \qquad \text{where } d = \log_2 A.  \]
\end{cor}

\begin{proof}
  We have that for $\frac{R+1}{2^k} \leq r+1$,
  \[ V(e,R) \leq V(e,R+1) \leq 2 V(e, \frac{R+1}{2}) \leq \cdots \leq A^k V(e, \frac{R+1}{2^k}) \leq A^k V(e, r+1) \leq A^{k+1} V(e, r),  \]
   After taking $\log_2 (\frac{R+1}{r+1}) \leq k \leq \log_2 (\frac{R+1}{r+1}) + 1$, we can deduce
  \[
    \frac{V(e,R)}{V(e,r)} \leq A^{\log_2 (\frac{R+1}{r+1}) + 2} = A^2 \left( \frac{R+1}{r+1} \right)^d.
    \qedhere
    \]
\end{proof}

To conclude this section, we show that $\cost{\cdot}$ also satisfies a ``reverse doubling'' property, i.e. a lower bound of $V(e,R)/V(e,r)$ by a quantity that is a polynomial of $R/r$.

\begin{lem}
  Let $G$ be a finite group, $\cost{\cdot}$ be the cost function of an $(S,\sp)$-long-jump random walk.
  Let $1 \leq R < \dia$, there exists $g \in G$ such that $R/4 \leq \cost{g} \leq R$.
  \label{lem:rings}
\end{lem}

\begin{proof}
  In the case that $1 \leq R \leq 4$, fix $s$ in $S$ that is not the identity. Then, $s \in B(e,R)$, and  $R/4 \leq \cost{s} \leq R$.
  Now consider $R$ such that $4 \leq R < \dia$. Since $R$ is strictly smaller than $\dia$, $G \setminus B(e,R)$ is non-empty and there exists $g \in B(e,R)$ such that $gs \in G \setminus B(e,R)$ where $s$ is an entry in $S$.
  Therefore, $R < \cost{gs} \leq 2( \cost{g} + \cost{s}) = 2 (\cost g + 1)$, and $\cost g > R/4$.
\end{proof}

\begin{prop}
    For all $1 \leq R \leq \dia$ and $r = 2^{-7} R$,
     \[ \frac{V(e,R)}{V(e,r)} \geq 2.\]
    Consequently, for all $1 \leq r \leq R \leq \dia$,
    \begin{equation} \frac{V(e,R)}{V(e,r)} \geq \frac 12 \left( R/r \right)^{1/7}.
      \label{eqn:reversedouble}
    \end{equation}
\end{prop}

\begin{proof}
  There are no elements with cost in $(0,1)$. So if $1 \leq R \leq 4$, $\# B(e,r) = 1$ and $\# B(e,R) \geq 2$.
  Now, we assume that $4 \leq R \leq \dia$. By Lemma \ref{lem:rings}, there exists $o \in B(e, R/4)$, such that $2^3 r =  R/2^4 \leq \cost o \leq R/2^2$.
  We will show that
  \begin{enumerate}
    \item $B(e,r) \cap B(o, r) = \emptyset$
    \item $B(e,r) \cup B(o, r) \subseteq B(e,R)$.
  \end{enumerate}
  This immediately implies that there are two disjoint balls of radius $r$ in $B(e,R)$, which is our desired result.
  To show (1), suppose there exists $g \in B(o,r) \cap B(e,r)$.
  By definition, we know that $\cost g \leq r$ and $\cost{o^{-1} g} \leq r$.
  This implies that $2^3 r \leq \cost o \leq 2(\cost{o^{-1} g} + \cost g) \leq 2^2 r$, which is a contradition.
  For (2), the fact that $B(e,r) \subseteq B(e,R)$ is clear. If $g \in B(o, r)$, then $\cost g \leq 2 (\cost o + \cost{o^{-1}g}) \leq 2 (R/4 + R/2^7) \leq R$.

  To show (\ref{eqn:reversedouble}), observe that for $ R \geq 2^{7k}r$,
  \[
  V(e,R) \geq 2 V(e,  2^{-7} R) \geq \cdots \geq 2^k V(e, 2^{-7k} R) \geq 2^k V(e,r).
  \]
  By choosing $k$ so that $ \frac 17 \log_2 (R/r) - 1  \leq k \leq \frac 17 \log_2 (R/r)$, we get the desired result.
\end{proof}

\section{Estimates on mixing and proof of Theorem \ref{thm:intro-mixing}}
\label{sect:mixing}

\begin{thm}
  \label{thm:l2mixing}
  Let $K$ be the Markov kernel of an $(S, \sp)$-long-jump random walk on a finite group $G$ with nilpotency $\ell$, and $\pi$ be the uniform distribution.
  There exists $b_1, b_2, c_1, c_2 > 0$ such that for all $n > 0$
  \[ \frac{c_1}{V(e,n)^{1/2}} \exp\left( - n/ b_1 \dia \right) \leq ||k_e^n - 1||_2  \leq \frac{c_2}{V(e,n)^{1/2}} \exp\left( - n/ b_2 \dia \right).  \]
\end{thm}

\begin{proof}[Proof of the upper bound of Theorem \ref{thm:l2mixing}]
   We have shown that there exists positive real numbers $d = d (\ell, k, \sp) \geq 1$ and $a = a(\ell, k, \sp) \geq 1$, for all $0 \leq r \leq R \leq \dia$ and $f \in \ell^2(\pi)$,
  \begin{equation}
      V(e, r) \geq  A^{-2} V(e,R) \left( \frac{r + 1}{R+1} \right)^d   \and || f - f_r ||_2^2 \leq a r \e_{\mus}(f,f).
     \label{eq:modpoin}
  \end{equation}
It is straighforward to check that the proof of \cite[Theorem 5.2]{DiaSal1996:459} works for quasi-norms.
Using  \cite[Remark 5.4 (2)]{DiaSal1996:459} with $\alpha=1$ and $M = \frac{A^2(R+1)^d}{V(e,R)}$, (\ref{eq:modpoin}) implies that
  \[
    \forall f \in \ell^2(\pi), \qquad
    ||f||_2 ^{2 + 2/d} \leq C \left( \e_{\mus}(f,f) + \frac{1}{a R^2} ||f||_2^2 \right) ||f||_1^{2/d},
  \]
  where $C = (1+1/(2d))^2 (1+2d)^{1/d} A^{2/d} (R+1) (V(e,R))^{-1/d} a$.
  This is called a \emph{Nash inequality} \cite{DiaSal1996:459}.
  By \cite[Corollary 3.1]{DiaSal1996:459} with $R = n$,
  we obtain that,
  \begin{equation}
      \forall n \leq \dia, \qquad ||K^{n}||_{2 \to \infty} \leq \frac{c_3}{V(e,n)^{1/2}},
      \label{eqn:2toinfty}
  \end{equation}
  where $c_3 = 2 \sqrt{2} (2^{d/2}) (1+\ceil{2d})^{2d} (1+1/(2d))^{d} (1+2d)^{1/2} a^{d/2} A$.

  Now, fix $n > 0$ and write $n=n_1+n_2$ with $n_1 = \min ( \floor{\dia}, n)$.
  We have
  \[
    ||k_e^n - 1 ||_2 = || K^n - 1 ||_{2 \to \infty} \leq ||K^{n_1}||_{2 \to \infty} || K^{n_2} - \pi ||_{2 \to 2},
  \]
  e.g. \cite[Section 1.2.4]{Saloff-Coste1997}.
  Inequality (\ref{eqn:2toinfty}) implies that $||K^{n_1}||_{2 \to \infty} \leq  \frac{c_3}{V(e,n)^{1/2}}, $
  and Theorem \ref{thm:intro-gap} and (\ref{eqn:betamin}) give
  \begin{align*}
    || K^{n_2} - \pi||_{2 \to 2} &\leq (1- \min(2\alpha_*,1/a\dia))^{n_2}
    \qquad (\text{where } \alpha_* = \min_{\alpha \in \sp} \frac{\alpha}{2(1+\alpha)} )\\
    &\leq (1 - 2 \alpha_* / a\dia) ^{n_2}  \qquad (\text{since } 2 \alpha_* \leq  1 \text{ and } a \dia \geq 1 )\\
    &\leq \exp (-2 \alpha_* n_2 / a \dia).
  \end{align*}
  It follows that
  \begin{align*}
    ||k_e^n - 1||_2 &\leq ||K^{n_1}||_{2 \to \infty}|| K^{n_2} - \pi||_{2 \to 2} \leq \frac{c_3}{V(e,n_1)^{1/2}} \exp (-2 \alpha_* n_2 / a \dia) \\
    &\leq \frac{c_2}{V(e,n_1)^{1/2}} \exp (- n / b_2 \dia)
    \qquad (\text{where } c_2 = e^{1/a} c_3 \text{ and } b_2 = a/ (2\alpha_*) ).
    \qedhere
  \end{align*}
\end{proof}

    For the lower bound, we will use this following lemma which orginates from  \cite[Proposition 2.3]{CouGri1997:133} and  \cite[Lemma 3.1]{SalZhe2016:4133}.
\begin{lem}
  For all $n \geq 0$, there exists a constant $C \geq 0$ such that
  \[ ||k_e^n||_2 \geq \frac{e^{-2C}}{V(e,n)^{1/2}}. \]
  \label{lem:Veigen}
\end{lem}

\begin{proof}
    Let $U = B(e,n)$ and $K_U(x,y) = K(x,y)$ when $x$ or $y$ are in $U$ and $K_U(x,y) =0 $ otherwise. Let $\beta_U$ is the largest eigenvalue of $K_U$ and $\phi_U$ be its associated eigenvector.
    The argument is based on the fact that as a consequence of Cauchy-Schwarz inequality, the function $n \mapsto \frac{||K^nf||_2}{||K^{(n-1)}f||_2}$ is decreasing.
    We have
    \begin{align*}
      ||k_e^n||_2 &= ||K^n||_{1 \to 2} = \max_{f \neq 0} \left\{ \frac{||K^n f||_2}{||f||_1} \right\} \\
      &= \max_{f \neq 0} \left\{ \frac{||K^n f||_2}{||K^{n-1} f||_2} \cdots \frac{||Kf||_2}{||f||_2} \frac{||f||_2}{||f||_1} \right\} \\
      &\geq \max_{f \neq 0} \left\{ \left( \frac{||Kf||_2}{||f||_2} \right)^{n-1} \frac{||f||_2}{||f||_1} \right\} \\
      &\geq \max_{\substack{f \neq 0\\ \text{supp}(f) \subseteq U}} \left\{ \left( \frac{||Kf||_2}{||f||_2} \right)^{n-1} \frac{1}{V(e,n)^{1/2}} \right\}
      \qquad (\text{by Cauchy-Schwarz and } U = B(e,n)) \\
      &\geq \left( \frac{||K\phi_U||_2}{||\phi_U||_2} \right)^{n-1} \frac{1}{V(e,n)^{1/2}} \\
      &\geq \left( \frac{||K_U \phi_U||_2}{||\phi_U||_2} \right)^{n-1} \frac{1}{V(e,n)^{1/2}} \qquad \text{($\phi_U$ is positive)}\\
      &= \beta_U^{n-1} \frac{1}{V(e,n)^{1/2}}.
    \end{align*}
    Consider the value
    \begin{align*}
        \beta_U &= \max_{\substack{f \neq 0, \text{supp} (f) \subseteq U \\ ||f||_2 = 1}} ||K_U f||_2
        = \max_{\substack{f \neq 0, \text{supp} (f) \subseteq U \\ ||f||_2 = 1}} \angle{K f,f}_\pi
        = 1 - \min_{\substack{f \neq 0, \text{supp}(f) \subseteq U \\ ||f||_2 = 1}} \e_{\mus}(f,f).
  \end{align*}
    Consider the test functions $\mathds 1_e$ and  $\zeta_+$, the function from (\ref{eqn:defzetaplus}) with $R=n$.
    From Lemma \ref{lem:test2} and (\ref{eqn:l2testfct}), there exist a constant $C = C(k, \ell, \sp)$ such that $\frac{\e_{\mus}(\zeta_+,\zeta_+)}{||\zeta_+||^2} \leq C/n $. Thus, we have
  \begin{align*}
      \beta_U &\geq 1 - \min \left\{ \frac{\e_{\mus}(\charf_e,\charf_e)}{||\charf_e||^2}, \frac{\e_{\mus}(\zeta_+,\zeta_+)}{||\zeta_+||^2} \right\} \\
      &\geq 1 - \min \left\{ \frac{\alpha_*}{8}, \frac{C}{n} \right\}.
    \end{align*}
    Collecting our lower bound on $||k_e^n||_2$ and $\beta_U$, we derive
    \begin{align*}
      ||k_e^n||_2 &\geq \left(1-\min\left\{ \frac{\alpha_*}{8}, \frac{C}{n} \right\} \right)^n \frac{1}{V(e,n)^{1/2}} \\
      &\geq \exp \left( - \min\left\{ \frac{\alpha_*}{4}, \frac{2C}{n} \right\} n \right) \frac{1}{V(e,n)^{1/2}} \\
      &\geq e^{-2C} \frac{1}{V(e,n)^{1/2}}. \qedhere
    \end{align*}

\end{proof}

\begin{proof}[Proof of lower bound of Theorem \ref{thm:l2mixing}]
    We have that $||k_e^n - 1||_2 \geq 2 || \mus^{(n)} - \pi ||_{TV} \geq \beta_1^n$. From Theorem \ref{thm:intro-gap}, we know that there exists $a > 0$ such that $\beta_1 \geq 1 - a/\dia$. We also have the bound $\beta_{min} \geq - 1 + \frac 18 \alpha_*$ by using test function $\charf_e$ in (\ref{eqn:gap}). Let $c = \min(a/\dia, \alpha_*/8)$, and we compute further
    \begin{align}
      ||k_e^n - 1||_2 &\geq (1 - c)^n
      \geq e^{-2an/\dia}, &\text{(since $0 \leq c \leq 1/2$.)}
      \label{eq:c}
    \end{align}
    Let $C > 0$ be the constant from Lemma \ref{lem:Veigen}.
    In the case that $V(e,n) \leq e^{-4C}/4$, we have $n \leq \dia$, and
    the term $\exp(-2an/\dia)$ is roughly constant, namely,
    \[e^{-2a} \leq \exp (-2an/\dia)  \leq 1. \]
    Hence, it follows from Lemma \ref{lem:Veigen} that
    \begin{align*}
        ||k_e^n - 1||_2 \geq ||k_e^n||_2 - 1 \geq \frac{e^{-2C}}{V(e,n)^{1/2}} - 1
      \geq \frac{e^{-2C}}{2 V(e,n)^{1/2}} \geq \frac{e^{-2C}}{2 V(e,n)^{1/2}} \exp (-2an/\dia).
    \end{align*}
    When $V(e,n) \geq e^{-4C}/4$, by (\ref{eq:c}), we have
    \begin{align*}      ||k_e^n-1||_2 &\geq \exp(-2an/\dia)
      \geq \frac{e^{-2C}}{2 V(e,n)^{1/2}} \exp(-2an/\dia).
    \end{align*}

    Thus, the lower bound is true for $c_1 = \exp(-2C)/2$ and $b_1 = 1/2a$.
  \end{proof}

  The proof for continuous time is similar and we have the following result. For the definition of $H_t$ and $h_t^e$, see the Notation section at the end of Section \ref{sect:intro}.
\begin{thm}
  Let $H_t$, $t > 0$, be the continuous time Markov kernel of an $(S, \sp)$-long-jump random walk on a finite group $G$ with nilpotency $\ell$, and $\pi$ be the uniform distribution.
  Then, there exists $a_1, a_2, c_1, c_2 > 0$ such that for all $t > 0$
  \[ \frac{c_1}{V(e,t)^{1/2}} \exp\left( - t / a_1 \dia \right) \leq ||h_t^e - 1||_2  \leq \frac{c_2}{V(e,t)^{1/2}} \exp\left( - t/ a_2 \dia \right).  \]
\end{thm}

We now have the ingredients to prove Theorem \ref{thm:intro-mixing}, the mixing time result.

\begin{proof}[Proof of Theorem \ref{thm:intro-mixing}]
  We want to show that there exists $a_1, a_2, b_1, b_2 > 0$ such that for all $n > 0$,
  \[ a_1 \exp(-n/b_1 \dia) \leq  ||K^n (e,\cdot) - \pi||_{TV} \leq a_1 \exp(-n /b_2\dia). \]
  The lower bound follows directly from the fact that $||K^n (x,\cdot) - \pi||_{TV} \geq \beta_1^n$, see \cite[Proposition 5.5]{Sal2004:263}. For the upper bound, we know from Theorem \ref{thm:l2mixing},
  \begin{equation}
    ||K^n (e,\cdot) - \pi||_{TV} \leq \frac 12 ||k_e^n - 1||_2 \leq   \frac{c_1}{V(e,n)^{1/2}} \exp\left( - n/ c_2 \dia \right),
  \end{equation}
  for some $c_1, c_2 > 0$. When $n \geq \dia$, $V(e,n)$ is equal to one and only the exponential term remains.
  The total variation is always bounded above by $2$, and when $n \leq \dia$ the exponential term bounded above by a constant, which gives us our upper bound. It thus follows that
  \[ a_1 (\log 2) \dia \leq t_{mix} \leq a_2 (\log 4) \dia. \qedhere\]
\end{proof}

\section{On computing the diameter}
\label{sect:dia}
\newcommand{\heisen}[3]
{\begin{bmatrix}
     1 & #1 & #3 \\
     0 & 1 & #2 \\
     0 & 0 & 1
 \end{bmatrix}}

As one would expect, computing $\dia$ for arbitrary groups, $S$, and $\sp$ is a difficult problem in general.
More surprisingly, even just on the cyclic group computing $\dia$ is still quite nuanced.
In Section \ref{sect:algcycle}, we give an exact formula for $\dia$ when the $S = (1, s)$ and arbitrary $\sp$ and use those results in our examples.
We start with a remark about the relationship between $\dia$ and the diameter of the Cayley graph:
\begin{rmk}
  Let $G$ be a finite group and $S = (s_1, s_2, \dotsc, s_k)$ be a $k$-tuple whose elements generate $G$.
  Recall from the introduction, $\mathcal{W}$ is the set of words generated by an alphabet $\mathcal{S} = \{\s_1^{\pm 1}, \s_2^{\pm 1}, \dotsc, \s_k^{\pm 1}\}$ generated from $S$. We define the following quantity, which is comparable to the diameter of the Cayley graph.
  \[ D_{\mathcal{S}} = \max_{g \in G} \left\{ \min_{w \in \mathcal{W}: g = \rho(w)} \max_{1 \leq i \leq k} |\deg_{\s_i} (w) | \right\}. \]
  Notice the following facts:
  \begin{enumerate}
    \item if $\sp = (\alpha, \dotsc, \alpha)$ for some $\alpha \in (0,2)$, then $\dia = D_{\mathcal{S}}^\alpha$, and
    \item if $c > 0$ and $\sp$ such that $\sp$ and $c \sp \in (0,2)^k$, then $\dia^c = D_{S,c\sp}$.
  \end{enumerate}
\end{rmk}

Now we are ready to present three examples in the vein of the example in the introduction, Example \ref{ex:intro}.
We will compute $\dia$ for $G = \Cn$ where $\sp$ and $N$ are fixed and $S$ is set to $(1,s)$ for various $s$ of the same order.
See Appendix \ref{sect:algcycle} for detailed computations.

\begin{exmp}[Simple variation of Example \ref{ex:intro}]
  Let $t$ be a postive integer larger than $5$ and $$ N = t(t^2+1)(t^2+2).$$
  We want to find $\dia$ for $G = \Cn$, $\sp = (\alpha, 1)$, $\alpha \in (0,2)$, and $S = (1, s)$, where $s = (t^2+1)(t^2+2)$.
  By Theorem \ref{thm:N1s-diameter},
  \begin{align*}
    \dia    &\asymp \min \{ N^\alpha, \max \{ N^{4 \alpha /5} , N^{1/5} \} \}.
    \end{align*}
    Breaking this into cases, we have
    \begin{align*}
      \small
      \dia &\asymp \left\{ \begin{array}{ll}
          N^\alpha        &\text{if } \alpha \in (0,1/5), \\
          N^{1/5}         &\text{if } \alpha \in [1/5, 1/4), \\
          N^{4\alpha/5}   &\text{if } \alpha \in [1/4,2).
              \end{array}
            \right.
    \end{align*}
    For this example, $s$ is of order $N^{4/5}$ where $s$ divides $N$, which is the same set up in Example \ref{ex:intro} from the introduction.
    \label{ex:Cn1}
\end{exmp}

\begin{exmp}
  Next, we still have $t > 5$, $N = t(t^2+1)(t^2+2)$, $G = \Cn$, and $\sp = (\alpha, 1)$ for some $\alpha \in (0,2)$.
  For this example, we pick $S' = (1, s')$, with $s' = t^2 (t^2 + 2)$ and we will compute $D_{S',\sp}$.
  As in the previous example, $s' \asymp N^{4/5}$, but $s'$ doesn't quite divide $N$.
  Dividing $N$ by $s'$ using the Euclidean algorithm, we get
  \begin{align*}
    N       & = t s' + r \qquad (\text{where } r = t^3 + 2 t)  \\
    s'      & = t r.
  \end{align*}
  Applying Theorem \ref{thm:N1s-diameter}, we have
  \begin{align*}
    D_{S',\sp}    &\asymp \min \{ N^\alpha, \max \{ N^{4 \alpha /5} , N^{1/5}\},
  \max \{ N^{3\alpha/5} , N^{2/5} \} \},
  \end{align*}
  which gives us what we got in Example \ref{ex:Cn1} with two more cases
  \begin{align*}
    \small
    D_{S',\sp}&\asymp \left\{
            \begin{array}{ll}
              N^\alpha        &\text{if } \alpha \in (0,1/5), \\
              N^{1/5}         &\text{if } \alpha \in [1/5, 1/4), \\
              N^{4\alpha/5}   &\text{if } \alpha \in [1/4,1/2), \\
              N^{2/5}         &\text{if } \alpha \in [1/2, 2/3), \\
              N^{3\alpha/5}   &\text{if } \alpha \in [2/3, 2).
            \end{array}
            \right.
  \end{align*}
  \label{ex:Cn2}
\end{exmp}

\begin{exmp}
  Again, we let $t > 5$, $N = t(t^2+1)(t^2+2)$, $G = \Cn$, $\sp = (\alpha, 1)$, $\alpha \in (0,2)$.
  We choose $S'' = (1, s'')$, with $s'' = (t^2+1)^2$, which does not divide $N$ ``even more'' than in the previous example.
  Specifically, dividing $N$ by $s''$ using the Euclidean algorithm terminates in three steps instead of two:
  \begin{align*}
    N       & = t s' + t^3 + t  \qquad (\text{where } r_1 = t(t^2 + 1))\\
    s''     & = t r_1 + r_2     \qquad (\text{where } r_2 = t^2 + 1)\\
    r_1     & = t r_2.
  \end{align*}
  Applying Theorem \ref{thm:N1s-diameter}, we get
  \begin{align*}
    D_{S'',\sp}  \asymp \min \{ N^\alpha, \max \{ N^{4 \alpha /5} , N^{1/5}\},
        \max \{ N^{3\alpha/5} , N^{2/5}\},
      \max \{ N^{2\alpha/5}, N^{3/5} \} \},
  \end{align*}
  and $D_{S'', \sp}$
  \begin{align*}
    \small
    D_{S'',\sp}  \asymp
    \left\{
      \begin{array}{ll}
        N^\alpha        &\text{if } \alpha \in (0,1/5), \\
        N^{1/5}         &\text{if } \alpha \in [1/5, 1/4), \\
        N^{4\alpha/5}   &\text{if } \alpha \in [1/4,1/2), \\
        N^{2/5}         &\text{if } \alpha \in [1/2, 2/3), \\
        N^{3\alpha/5}   &\text{if } \alpha \in [2/3, 1), \\
        N^{3/5}         &\text{if } \alpha \in [1, 3/2), \\
        N^{2\alpha/5}   &\text{if } \alpha \in [3/2, 2).
    \end{array}
    \right.
  \end{align*}
  \label{ex:Cn3}
\end{exmp}

\begin{rmk}
  In the three examples above, $N$ is the same and $s, s'$ and $s''$ are comparable in size ($N^{4/5}$), but the resulting diameters $\dia$ change according to the length of the Euclidean division of $N$ by $s, s'$ or $s''$.
\end{rmk}

Next we give $\dia$ for a non-abelian group for different sets of generators.
\begin{exmp}
\label{ex:h3-simple}
 Let $G = H_3(\Cn)$ be the group of upper triangular matrices in $M_{3\times 3}(\Cn)$ with $1$'s on the diagonal.
 Let $g$ be a element of $H_3(\Cn)$, which we will write of the form
 \begin{equation} \heisen{x}{y}{z}.\label{eqn:h3-gform} \end{equation}
Let
\[
  s_1 = \heisen 1 0 0 \qquad s_2 = \heisen 0 1 0 \qquad \text{and} \qquad s_3 = \heisen 0 0 1.
\]
Consider $\dia$ with $S = (s_1, s_2, s_3)$ and $\sp = (\alpha_1, \alpha_2, \alpha_3)$, where each $\alpha_i \in (0,2)$.
Then,
\begin{equation}
    \cost{g} \asymp \max \left\{  \cl{x}^\aI, \cl{y}^\aII, \min \left\{ \cl{z}^{\alpha_3}, |z|^{\frac{\aI \aII}{\aI + \aII}} \right\} \right\}.
  \label{eqn:h3dia}
\end{equation}
Therefore, $\dia \asymp N^{\max \left\{ \alpha_1,\alpha_2,\alpha_3 \right\} }$.
\end{exmp}

If we also include $s_1^t$ in $S$, this decreases the cost of elements in both the $s_1$ and $s_3$ direction.
\begin{exmp}
  \label{ex:h3-mixed-fewer}
  Fix $t > 0$, and $N = t^2$.
  Let $G = H_3(\Cn)$,  $S = (s_1, s_1', s_2, s_3)$, $s_1' = s_1^t$, and $\sp = (\aI, \aI, \aII, \aIII)$. Let $g = s_3 ^{m_3} s_2 ^{m_2} s_1 ^{m_1} $. Define $x(m)$ and $y(m)$ so that $m = y(m) t + x(m)$ where $|x(m)| \leq t/2$, and therefore, $|y(m)| \leq t$,
  Then
  \[
      \cost{g} \asymp \max \left\{
      \max \{ |x(m_1)| , |y(m_1)| \}^\aI,
      |m_2|^\aII,
      \min \left\{ |m_3|^\aIII, \max \{ |x(m_3)|, |y(m_3)|\}^{\frac{\aI \aII}{\aI+\aII}} \right\} \right\}.
  \]
  Therefore, $\dia \asymp N^{\max \left\{ \frac{\alpha_1}{2}, \aII, \aIII, \frac{\alpha_1\alpha_2}{2(\alpha_1+\alpha_2)} \right\} }$.
\end{exmp}

\section{Generalizing results to $\sp \in (0, \infty)^k$.}
\label{sect:difalphas}

In this section, we discuss how to generalize the main results of the paper (Theorem \ref{thm:intro-gap}, Theorem \ref{thm:intro-mixing}, and Theorem \ref{thm:l2mixing}) when $\sp \in (0,\infty)^k$.

\begin{defn}
  For any $\alpha > 0$, define a function $\Phi_\alpha: \Cn \to \R$ as follows.
  \begin{equation}
    \Phi_\alpha(x) = \begin{cases}
      \cl{x}^\alpha      &\text{if $\alpha \in (0,2)$} \\[0.5em]
      \cl{x}^2 / \log \cl{x}  &\text{if $\alpha = 2$} \\[0.4em]
      \cl{x}^2           &\text{if $\alpha > 2$}
    \end{cases}.
    \label{eqn:costallalpha}
  \end{equation}
\end{defn}
We redefine the cost function from (\ref{eqn:cost}) as follows.
\begin{defn}
  For $g \in G$,
  \begin{equation}
    \cost{g} = \min_{ \substack{ w \in \W: \\ \rho(w) = g }} \left\{ \max_{1 \leq i \leq k} \left\{ \Phi_{\alpha_i} (\deg_{\s_i} (w)) \right\} \right\},
  \end{equation}
  where $\W$ is is the set of words generated by the alphabet $\mathcal S = \{ \s_1^\pm, \dotsc, \s_k^\pm \}$, and $\rho$ is the canonical projection from $\W$ to $G$.
We redefine the $(S,\sp)$-diameter with respect to the new cost function $$\dia = \max_{g \in G} \cost{g}. $$
\end{defn}

\begin{thm}
  Fix $C_0 > 0$ and  $0 < \ep < 1$. There exist $c_1, c_2 > 0$ (depending on $C_0$ and $\ep$) such that
  for any $\ell, k < C_0$, $\sp \in (\ep, 1/\ep)^k$, and any finite nilpotent group $G$ of nilpotency class $\ell$ generated by a $k$-tuple $S$,
  the $(S,\sp)$-long jump random walk satifies
   \[ c_1 / \dia \leq 1 - \beta_1 \leq c_2/\dia. \]
   \label{thm:gen-gap}
\end{thm}

To prove this we need the following lemma:

\begin{lem}
  Fix $N > 0$, $\alpha > 0$, and $G = \Cn$. Then, there exists $C(\alpha) > 0$ so that for all $r>0$, $\Phi_\alpha (y) \leq r$ and $f \in \ell^2(\pi)$,
  \begin{equation}
   \frac1{N} \sum_{x \in \Cn} |f(x) - f(x + y )|^2 \leq
   C(\alpha) \Phi_\alpha(y) \e_{ p_{N,\alpha} } (f,f),
  \end{equation}
    \label{thm:poincareoncycle-extended}
  where
  \[ p_{N,\alpha} (x) = \frac{c}{(1+\cl{x})^{1+\alpha}}
      \and c^{-1} = \sum_{j \in \Cn} \frac{1}{(1+\cl{j})^{1+\alpha}}.
 \]
\end{lem}

\begin{proof}
    When $\alpha > 2$, both $\e_p$ and $\e_\pua$ have finite second moment, where $p$ is the measure that drives lazy simple random walks on $\Cn$. Therefore,
    the two forms are comparable up to a constant, see \cite[Corollary 1.5]{pittet2000stability}. The case when $\alpha = 2$ is treated in \cite[Proposition A.4]{SalZhe2016:4133}.
\end{proof}

\begin{proof}[Proof of Theorem \ref{thm:gen-gap}]
    For the lower bound, we repeat the argument from Section 2. The appropriate version of Theorem \ref{thm:rewritingwords} comes from \cite[Theorem 2.10]{SalZhe2015:1047}. Lemma \ref{thm:poincareoncycle-extended} is the corresponding version of (\ref{eqn:orbitp}). From there, the proof follows the same line of reasoning.

For the lower bound, adapting details of Section 2.2 is a straightfoward calculus exercise. The main details of the computation is also covered in \cite[Lemma 4.2]{CheKum2018}.
\end{proof}

Still following the same reasoning as for $\sp \in (0,2)^k$,
we arrive to the following theorem.

\begin{thm}
  \label{thm:l2mixing-gen}
  Let $K$ be the Markov kernel of an $(S, \sp)$-long-jump random walk on a finite group $G$ with nilpotency $\ell$, and $\pi$ be the uniform distribution.
  There exists $b_1, b_2, c_1, c_2 > 0$ such that for all $n > 0$
  \[ \frac{c_1}{V(e,n)^{1/2}} \exp\left( - n/ b_1 \dia \right) \leq ||k_e^n - 1||_2  \leq \frac{c_2}{V(e,n)^{1/2}} \exp\left( - n/ b_2 \dia \right),  \]
  where
  \[
      V(x,r) = \sum_{y \in G: \cost{x^{-1}y} \leq r} \pi(y).
  \]
\end{thm}

\begin{exmp}
  Fix $t > 0$, and let $N = t^2$, $G = \Cn$, $S = (1, t)$, and $\sp = (1, 2)$.
  For each $g \in \Cn$, we can write $g = x_1 + x_2 t$ so that $|x_1|$ and $|x_2|$ are strictly less than $t$.
  Then,
  \begin{align*}
    \cost{g} = \cost{x_1 + x_2 t} \asymp \max \left\{ |x_1|, \frac{|x_2|^2}{\log |x_2|} \right\}.
  \end{align*}
\end{exmp}

\newpage
\appendix

\section{A note on properties of $\mus$}
\label{sect:mus-equivalence}

In this section, we prove some useful lemmas about $\mus$ that are used frequently throughout the paper, such as bounds for the normalization constants and its moments. We will also discuss how $\mus$ compares to the probability measure that drives the wrap-around model.

\begin{lem}
  Let $N$ be a positive integer, $\alpha$ be a positive real number, and
  \[c = \left( \sum_{j \in \Cn} \frac{1}{ (1 + \cl j)^{1+\alpha}} \right)^{-1}, \]
  the normalization constant on individual cycles from Definition \ref{defn:longjump}.
  Then,
  \[ \frac\alpha{2(1+\alpha)} \leq c \leq 1.\]
  \label{lem:cna_indep}
\end{lem}

\begin{proof}
   Since the summand corresponding to $j = 0$ is $1$, $c$ is less than or equal to $1$.
   For the lower bound, we have that for all $N \geq 1$,
    \begin{align*}
      c^{-1} &\leq  2 \left( 1 + \sum_{k=1}^{N/2} \frac{1}{(1+k)^{1+\alpha}}\right)
      \leq 2 \left( 1 + \int_{0}^{N/2} \frac{1}{(1+s)^{1+\alpha}} \; ds\right)
      \leq 2 \left( 1 + \frac{1}{\alpha} \right).
    \end{align*}
\end{proof}

\begin{lem}
  Let $p:\Cn \to [0,1],$ $p(j) = c/(1+\cl{j})^{1+\alpha}$ with $\sum_{j \in \Cn} p(j) = 1$.
  There exists a constant $C(\alpha) > 0$, so that
\[ \sum_{\substack{|t| > a }} p (t) \leq  \frac{C(\alpha)}{a^{\alpha}} \and
  \sum_{\substack{|t| < a }} |t|^2 p (t) \leq C(\alpha) a^{2 - \alpha}. \]
  \label{lem:upperpi}
\end{lem}

\begin{proof}
  If $a = 1$, then both sums are less than or equal to $1 = 1/a^{\alpha_i} = a^{2-\alpha}$.
  If $a > 1$, we can compute
  \begin{align*}
    \sum_{|t| > a} p (t) &\leq 2 \int_{\max(2,a)-1}^\infty p(t) \; dt
    \leq 2c \int_{a/2}^\infty \frac{dt}{(1+t)^{1+\alpha}} \\
    &= \frac{2c}{\alpha (1+a/2)^{\alpha}}
    \leq \frac{2c (2^{\alpha})}{\alpha} \frac{1}{a^{\alpha}}.
  \end{align*}
  Moreover,
  \begin{align*}
    \sum_{|t| < a} |t|^2 p (t) &\leq 2 \int_0^{4a} t^2 p(t) \; dt
    \leq 2c \int_0^{4a} (1+t)^{1-\alpha} \; dt \\
    &\leq \frac{2^5c}{2-\alpha}a^{2-\alpha}
  \end{align*}
  Thus, the statement of the lemma is true for $C(\alpha) = \max\{1, 2 c (2^{\alpha}) / \alpha, 2^5c/(2-\alpha) \}$.
\end{proof}

Next, we prove that the wrap-around definition described in the introduction and one given in Definition \ref{defn:longjump} are comparable.
Specifically, we will show that on cycles, the probability measures are comparable up to multiplicative constants depending only on $\alpha$. Therefore, the probability measures on the full group are comparable up to constants depending on $\sp$.
For fixed $N > 0$ and $\alpha \in (0,2)$, the measure driving our long-jump random walks on $\Cn$ is
\begin{align*}
  p(g) &= \frac{c}{(1+\cl g)^{1+\alpha}},
    \qquad \text{where } c^{-1} = \sum_{j \in \Cn} \frac{1}{(1+\cl j)^{1+\alpha}},
\end{align*}
and the measure driving the wrap-around model is
\begin{align*}
  \tilde p (g) = \sum_{j \in \Z} \frac{\tilde c}{(1+|g + N j|)^{1+\alpha}},
\qquad \text{where } {\tilde c}^{-1} = \sum_{g \in \Cn} \left( \sum_{j \in \Z} \frac{1}{(1+|g + Nj|)^{1+\alpha}} \right).
\end{align*}
\begin{lem}
  For all $\alpha > 0$, there exist constants $c_1, c_2 > 0$ depending on $\alpha$, such that for all positive integer $N$ and $k \in \Cn$,
  \[ c_1 p(k) \leq \tilde p(k) \leq c_2 p(k). \]
\end{lem}

\begin{proof}
  Fix an integer $k \in [0,N/2]$, and consider
  \begin{align}
    \frac{\tilde p(k)}{p(k)} &= \frac{\tilde c}{c} \sum_{j \in \Z} \frac{(1+k)^{1+\alpha}} {(1 + |k+Nj|)^{1+\alpha} }
    = \frac{\tilde c}{c} \left( 1 + \sum_{j \neq 0}  \frac{(1+k)^{1+\alpha}} {(1 + |k+Nj|)^{1+\alpha}} \right)
    \label{eq:ratio}
  \end{align}
  It will be convenient to define the constant $A = \sum_{j = 1}^\infty \frac{1}{ (1+j)^{1+\alpha} }$.
  For the lower bound of (\ref{eq:ratio}), notice that the term in the parenthesis is bounded below by $1$. Moreover,
  since rearranging the summand gives that $\tilde c^{-1} \leq 2A$,
  combined with Lemma \ref{lem:cna_indep}, we see that we can set $c_1 =\frac{2(1+\alpha)}{\alpha A}$.

  For the upper bound of (\ref{eq:ratio}), first when $k = 0$,
  for all $k \in [0,N/2]$ and $j > 0$,
  \begin{align*}
    \frac{1+k}{1+k+j N} &= \frac{1}{1 + j \frac{N}{k+1}} \leq \frac{2}{1+j},
  \end{align*}
  where the last inequality is because $k+1 \leq 2N$.
  When $j < 0$, we have
  \begin{align*}
    \frac{1+k}{1-j N+k} \leq \frac{1}{1+|j| \frac{N}{k+1}} \leq \frac{2}{1+|j|}.
  \end{align*}
  Then,
  \[\sum_{j \neq 0}  \frac{(1+k)^{1+\alpha}} {(1 + |k+Nj|)^{1+\alpha}} \leq \sum_{j \neq 0} \left( \frac{2}{1 + |j|} \right)^{1+\alpha} \leq 2^{2+\alpha} A. \]
  Thus, we can set $c_2 = (1+2^{2+\alpha})A / c$.
\end{proof}

\section{Dirichlet form estimates}
\label{sect:dirform}
In this section, we establish various estimates on the Dirichlet form. We will also let $\pi$ always be the uniform distribution for simplicity of proof, but all theorems can be made to work for arbitrary distributions. The techniques in this section were inspired by techniques developed in \cite{SalZhe2015:1047}, \cite{SalZhe2015:837}, and \cite{SalZhe2016:4133}. In particular, see Section 4 of \cite{SalZhe2015:837}.

\begin{prop}
  \label{prop:dirprops}
  Let $G$ be a finite group, $||\cdot||$ a quasi-norm on $G$, $\mu$ a probability measure on $G$, and $\pi$ be the uniform distribution.
  \begin{enumerate}
    \item Suppose that there exists a function $a(r) \geq 0$, such that for all $r \geq 0$, $f \in \ell^2(\pi)$, $y \in B(0,r)$,
      \[\sum_{x \in G} |f(x) - f(xy)|^2 \pi(x) \leq a(r) \e_\mu(f,f).\]
      Then for all $r \geq 0$,  $f \in \ell^2(\pi)$,
      \[||f-f_r||^2_2 \leq a(r) \e_\mu(f,f). \]
    \item Fix $s \in G$ and $n$ to be the order of $s$ in $G$. Let $\mu$ is a probability distribution on $G$ of the form
      \[\mu(g) = \sum_{j \in \Cn } \charf_{s^j} (g) p (j) , \]
      where $p$ is a probability distribution on $\Cn$.
      Let also that $||\cdot||$ be a quasi-norm of the form
      \[ ||g|| = \begin{cases}
      ||m||_0 & \text{if $g = s^m$}\\
      \infty &\text{otherwise}
       \end{cases}, \]
       where $||\cdot||_0$ is a quasi-norm on $\Cn$.

       Suppose that there exists a real-valued function $a(r) \geq 0$ such that for all $r \geq 0$, $f: \Cn \to \R$, $y \in \Cn$ where $||y||_0 \leq r$,
       \[ \frac{1}{N} \sum_{x \in \Cn} |f(x) - f(x+y)|^2 \leq a(r) \e_{p}(f,f). \]
      Then we have that for all $r \geq 0$, $f: G \to \R$, $y \in G$ where $||y|| \leq r$ then
       \[ \frac{1}{|G|} \sum_{x \in G} |f(x) - f(xy)|^2 \leq a(r) \e_{\mu}(f,f). \]
    \item Let $\mu: G \to \R$ be a convex combination of probability measures $\mu_i$: $\mu = \sum_{i=1}^k c_i \mu_i.$
      Then for any $f \in \ell^2(\pi)$, then
      \[ \e_{\mu} (f,f) = c_1 \e_{\mu_{1}}(f,f) + \dotsb + c_k \e_{\mu_{k}}(f,f).\]
  \end{enumerate}
\end{prop}

\begin{proof}
  \begin{enumerate}
    \item For all $r \geq 0$ and $f \in \ell^2(\pi)$, we have
  	  \begin{align*}
  	    ||f-f_r||_2^2 &\leq \sum_{x \in G} |f(x) - f_r (x)|^2 \pi (x) \\
  	    &= \sum_{x \in G} \left| \frac{1}{V(e,r)} \sum_{y \in B(e,r)} (f(x) - f(xy)) \pi (xy) \right|^2 \pi (x) \\
  	    &= \sum_{x \in G} \left| \sum_{y \in B(e,r)} (f(x) - f(xy)) \frac{1}{\# B(e,r)} \right|^2 \pi (x) \\
  	    &\leq \sum_{x \in G} \frac{1}{V(e,r)} \sum_{y \in B(e,r)} \left| (f(x) - f(xy))  \right|^2 \pi (x) \qquad \text{(by Jensen's inequality)}\\
  	    &\leq \frac{1}{\# B(e,r)} \sum_{y \in B(e,r)} a(r) \e_\mu(f,f) \qquad \text{(by assumption)}\\
        &= a(r) \e_\mu(f,f).
  	  \end{align*}
    \item   Fix $r > 0$, $f: G \to \R$ and $y_0 \in G$ such that $||y_0|| \leq r$. By the definition of $||\cdot||$, $y_0$ is of the form $s^m$ where $||m||_0\leq r$.
      We denote the cosets of $\langle s \rangle \leq G $ as $[x_j \langle s \rangle]$ where the $x_j$'s are fixed representatives of the cosets.
      Then we have
      \begin{align*}
          \e_{\mu }(f,f)
          &= \frac{1}{|G|} \sum_{x,y \in G} |f(x) - f(xy)|^2 \mu (y) \\
          &= \frac{1}{|G|} \sum_{x,y \in G} |f(x) - f(xy)|^2 \sum_{j \in \Cn} \charf_{s^j}(y) p (j) \\
          &= \frac{1}{|G|} \sum_{x \in G} \sum_{j \in \Cn} |f(x) - f(x s^j)|^2 p (j).
      \end{align*}
      For each $x$ there is an unique representation as a product of one of the $x_j$'s and an element in $\langle s \rangle$.
      So we have
      \begin{align*}
        \e_{\mu }(f,f) = \frac{N}{|G|} \sum_{j = 1}^{|G|/N} \sum_{\ell, \ell' \in \Cn} |f(x_j s^{\ell'}) - f(x_j s^{\ell+\ell'})|^2 p (\ell) \frac{1}{N}.
      \end{align*}
      Then define $f_j: \Cn \to \R$ to map $\ell \mapsto f(x_j s^{\ell})$, and we have
      \begin{equation}
           \e_{\mu}(f,f) = \frac{N}{|G|} \sum_{j = 1}^{|G|/N} \e_{p} (f_j, f_j).
           \label{eqn:breakupeform}
      \end{equation}
     Then
      \begin{align*}
        \frac{1}{|G|} \sum_{x \in G} |f(x) - f(x s^m)|^2 &= \frac{1}{|G|} \sum_{j = 1}^{|G|/N} \sum_{\ell \in \Cn} |f(x_j s^\ell) - f(x_j s^{\ell+m})|^2 \\
        &= \frac{1}{|G|} \sum_{j = 1}^{|G|/N} \sum_{\ell \in \Cn} |f_j(\ell) - f_j(\ell+m)|^2 \\
        &\leq \frac{N}{|G|} a(r) \sum_{j = 1}^{|G|/N} \e_\mu(f_j,f_j) \qquad \text{(by assumption)} \\
        &= a(r) \e_{\mu}(f,f) \qquad \text{(by (\ref{eqn:breakupeform}))}.
      \end{align*}
    \item   \begin{align*}
          \e_{\mu} (f,f) &= \sum_{x,y} |f(x) - f(xy)|^2 \mu (y) \pi(x) \\
          &= \sum_{x,y} |f(x) - f(xy)|^2 (c_1 \mu_{1}(y) + \dotsb + c_k \mu_{k}(y)) \pi (x) \\
          &= c_1 \sum_{x,y} |f(x) - f(xy)|^2 \mu_{1}(y) \pi(x) + \dotsb + c_k \sum_{x,y} |f(x) - f(xy)|^2 \mu_{k}(y) \pi(x) \\
          &= c_1 \e_{\mu_1}(f,f) + \dotsb + c_k \e_{\mu_k}(f,f)
      \end{align*}
  \end{enumerate}
\end{proof}

\begin{defn}\label{defn:regularitycondi}
   Define $p$ a symmetric distribution on $\Cn$ to satisfy \emph{regularity condition (A)} if there exists a constant $C_p > 0$ such that
   for all $k \in [0,N/2]$
   \[
        \min_{I_k} p \geq C_p \max_{I_k} p,
   \]
    where $I_k = [\floor{k/{\color{red}}  9}, k]$. Since $p$ is symmetric the inequality remains true for $k \in [-N/2, 0]$ with $I_k = [k, \ceil{k/9}]$.
\end{defn}

\begin{lem}
  Let $N \geq 0$ and  $\alpha > 0$. The probability distribution $\pua: \Cn \to \R$, where
  \begin{equation}
    \pua(x) = \frac{\cna}{(1+|x|)^{1+\alpha}} , \qquad
    \text{and }\cna^{-1} = \sum_{j \in \Cn} \frac{1}{(1+|j|)^{1+\alpha}}.
    \label{eqn:pua}
  \end{equation}
  satisfies regularity condition (A)
  where the constant $C_{\pua}$ depends only on $\alpha$, and not $N$.
  \label{lem:regularity-of-power-law}
\end{lem}

\begin{proof}
  By Lemma \ref{lem:upperpi}, we know that
    \[ \frac {\alpha}{2 (1+\alpha) } \leq \cna \leq 1.\]

  Let $k \in [0,N/2]$.
  In the trivial case, when $k \in [0,8]$, $I_k$ only includes $0$, so we have $\min_{I_k} \pua = \max_{I_k} \pua = \cna$.

  In the nontrivial case, when $k \geq 9$, we have $\floor{k/9} \geq \max \{ 1 , k/9 - 1 \}$.

  Then we have
    \begin{align*}
      \max_{I_k} \pua
        &= c (1 + \floor{k/9})^{- (1 + \alpha)}
  	    \leq (1 + \max \{1, k/9-1\})^{-(1+\alpha)} \\
        &\leq 2^{1+\alpha} (2 + 1 + k/9 - 1)^{- (1 + \alpha)}
        \leq 18^{1+\alpha} (18 + k)^{- (1+\alpha)}\\
        &\leq 18^{1+\alpha} (1 + k)^{- (1+\alpha)}
        \leq 18^{1+\alpha} \min_{I_k} \pua.
    \end{align*}
  Thus, $C_\pua$ can be set to $18^{-1-\alpha}$.
\end{proof}

Next we show that that the pseudo-Poincar\'e inequality holds for $\pua$, as defined in (\ref{eqn:pua}), on the cyclic group.

\begin{thm}
  Fix $\alpha \in (0,2)$ and $n > 0$. Then there exists $C(\alpha)$ so that for all $r>0$, $|y|^\alpha < r$ and $f \in \ell^2(\pi)$
  \begin{equation}
   \frac1{N} \sum_{x \in \Cn} |f(x) - f(x + y )|^2 \leq
   C(\alpha) |y|^\alpha \e_{\pua} (f,f).
   \label{eqn:e-baby}
  \end{equation}
  \label{thm:poincare-circle-02}
\end{thm}

\begin{proof}
The statement is trivially true when $y = 0$.
And for $y \neq 0$, we first define $$I_0 = \begin{cases}[ \floor{y/4}, y/2] &\text{if $y\geq 0$} \\ [y/2, \ceil{y/4}] &\text{if $y \leq 0$} \end{cases}.$$
Note that $I_0$ is always non-empty, since if $|y| \in \{1,2,3\}$, then $0 \in I_0$. For all other $y$'s, $\floor{|y|/4}$ and $\floor{|y|/2}$ are at least one apart.
Then, first multiplying the left hand side of (\ref{eqn:e-baby}) by $\pua(y)$, we create two sums $A$ and $B$:
\begin{align*}
    \sum_{x \in G} |f(x) - f(x+y)|^2 \pua(y)
    \leq \frac{1}{|I_0|} \Bigg(
    &\underbrace{\sum_{z \in I_0} \sum_{x \in G} |f(x) - f(x+z)|^2\pua(y)}_A \\
  &+ \underbrace{\sum_{z \in I_0} \sum_{x \in G} |f(x+z) - f(x+y)|^2 \pua(y) }_B \Bigg).
\end{align*}

Define $$J_y = \begin{cases} [\floor{y/9}, y] &\text{if $y \geq 0$} \\ [y, \ceil{y/9}] &\text{if $y < 0$} \end{cases}.$$ By the regularity property (Definition \ref{defn:regularitycondi}) of $\pua$ and the fact that $\pua$ is symmetric and $I_0 \subseteq J_{y}$, we have
\begin{align*}
    &\pua(y) \leq \frac{1}{C_\pua} \pua(z) &\text{for all $y \in I_0$}
\end{align*}
Thus,
\[ A \leq \frac{1}{C_\pua} \sum_{x \in \Cn} \sum_{z \in I_0} |f(x) - f(x+z)|^2 \pua (z) \leq \frac{N}{C_\pua} \e_\pua(f,f).\]

Moreover, since $z \in I_0$, we have
$$ y-z \in \begin{cases}[ \floor{y/4}, y/2 ] &\text{if $y \geq 0$} \\ [y/2, \ceil{y/4}] &\text{if $y <0$}\end{cases} .$$
So
\[ \pua(y) \leq \frac{1}{C_\pua} \pua(y - z),\]
and
\[ B \leq \frac{1}{C_\pua} \sum_{z \in I_0} \sum_{x \in \Cn} |f(x+z) - f(x+y)|^2 \pua(y-z) \leq \frac{N}{C_\pua} \e_\pua(f,f). \]

Then combining what we computed, we have
\begin{align*}
  \frac 1 N \sum_{x\in \Cn} |f(x) - f(x+y)|^2 \leq \frac{2}{C_\pua \pua(y) \#I_0 }\e_\pua(f,f) \leq
    \frac{4 |y|^{1+\alpha}}{\#I_0 \cna C_\pua} \e_\pua (f,f).
\end{align*}
where the last inequality is by:
\[
    \begin{cases}
        (1+ |y|)^{1+\alpha} \leq 2^{1-\alpha} |y|^{1+\alpha}
        & \text{if } y \neq 0 \\
        (1+ |y|)^{1+\alpha} \leq |y|^{1+\alpha}
        &\text{if } y = 0
    \end{cases}.
\]

Then to count the number of elements in $I_0$, we see that when $|y| = 1$, $I_0$ has one element;
when $|y| = 2,3,4$, $I_0$ has 2 elements, and for $|y| \geq 8$, we have
\[ \#I_0 \geq \floor{\frac{|y|}{2} - \floor{|y|/4}} \geq \floor{|y|/4} \geq |y|/4 - 1 \geq |y|/8.\]
And for $4 < |y| <8$, we have that $|y|/4$ is one, and $|y|/2 > 2$, so $\#I_0 \geq 2$.
In all cases, the $\#I_0 \geq |y|/8$.

Therefore if we set, using previous bounds for $\cna$ and $C_\pua$,  $C(\alpha)$ to $2^{9+\alpha}3^{2+\alpha} (\alpha+1) /\alpha$, then
\[ \frac{1}{N} \sum_{x\in \Cn} |f(x) - f(x+y)|^2 \leq C(\alpha) |y|^\alpha \e_\mu(f,f),\]
which is what we were looking for in (\ref{eqn:e-baby}).
\end{proof}

By Theorem \ref{thm:poincare-circle-02}, Proposition \ref{prop:dirprops} (2), and the definition of $||\cdot||_{s,\alpha}$, we obtain the following theorem.

\begin{thm}
  Let $G$ be finite group, $s \in G$, and $\alpha \in (0,2)$. Then as defined in the end of Section \ref{para:notation},
  \begin{equation}
    \mual (g) =\sum_{\ell \in \Z/N_i \Z} \charf_{s_i^\ell}(g) p_i(\ell).
  \end{equation}
  There exists a constant $C(\alpha) > 0$ such that for all $r \geq 0$, $f \in \ell^2(\pi)$, and $y \in G$ where $||y||_{s,\alpha} \leq r$,
\[ \sum_{x \in G} |f(x) - f(xy)|^2 \pi(y) \leq C(\alpha) r \e_\mual (f,f),\]
where $C(\alpha)$ can be defined as $2^{9+\alpha}3^{2\alpha}(1+\alpha) /\alpha$.
\label{thm:embedded}
\end{thm}

At this point, it may be illustrative to use the above theorem to prove a pseudo-Poincar\'e inequality for finite abelian groups.
Let $G$ be a finite abelian group, $S$ be a $k$-tuple of generating elements of $G$, and $\sp \in (0,2)^k$.
Fix $r > 0$, $f: G \to \R$, $y \in G$ where $\cost{y} \leq r$.
Then $y$ can be written as $y = y_1 y_2 \dotsb y_k$ so that for all $i$, $||y_i||_{s_i, \alpha_i} \leq r$. So by Theorem \ref{thm:embedded}, for all $1 \leq i \leq k$,
\[\sum_{x \in G} |f(x) - f(xy_i)|^2 \pi(x) \leq C(\alpha_i) r \e_{\mu_{s_i,\alpha_i}} (f,f), \]
where $C(\alpha_i)$ is the number defined in Theorem \ref{thm:embedded}.
Using these inequalities we have
\begin{align*}
  \frac{1}{|G|} \sum_{x \in G} |f(x) - f(xy)|^2 & = \frac{1}{|G|} \sum_{x \in G} |f(x) - f(xy_1 \dotsb y_k)|^2 \\
  &\leq  k \sum_{i=1}^k \frac{1}{|G|} \sum_{x \in G} |f(x) - f(x y_i)|^2  \quad \text{(by Cauchy-Schwarz inequality)}\\
  &\leq C(\sp) r k \sum_{i=1}^k \e_{\mu_{s_i, \alpha_i}} (f,f) \quad (\text{using Proposition \ref{prop:dirprops}}) \\
  &= C(\sp) r k^2 \e_{\mus} (f,f),
\end{align*}
where $C(\sp) = \max_{1 \leq i \leq k} C(\alpha_i)$.

\definecolor{bleu}{RGB}{80,80,160}
\definecolor{verte}{RGB}{80,160,80}

\section{Algorithm for computing $\dia$ for cyclic groups}
\label{sect:algcycle}

The goal of this section is to give an algorithmic way to compute $\dia$ on the cyclic group $\Cn$ with $S = (1 , s)$
and $\sp = (\aI, \aII).$
We use this process to arrive at the examples outlined in Section \ref{sect:dia}.
For convenience, we will think of elements of $\Cn$ as integers in $\{0, \dotsc, n-1\}$, and fix $1 \leq s \leq N/2$.
We know that for all positive integers $0 < a \leq b/2$,  there exists positive integers $q$ and $r$ such that
\[
  b = q a - \ep r,
\]
such that $\ep \in \{\pm 1\}$ and $0 \leq r \leq a/2$.

Using this fact to modify the Euclidean algorithm, we can expand $N$ as follows:
\begin{align}
  r_{-1} = N    &= q_1 s    - \ep_1 r_1 \nonumber \\
  r_0 = s       &= q_2 r_1  - \ep_2 r_2 \nonumber \\
  r_1           &= q_3 r_2  - \ep_3 r_3 \nonumber \\
                &\;\; \vdots            \nonumber \\
  r_{i-1}       &= q_{i+1} r_i - \ep_{i+1} r_{i+1} \nonumber \\
                &\;\; \vdots           \nonumber \\
  r_{K-1}       &= q_{K+1} r_K - r_{K+1}
  \label{eqn:euclidalg}
\end{align}
where $r_{K+1}$ is the first $r_i$ that's equal to zero, so $r_K$ is equal to the greatest common divisor of $N$ and $s$.
The connection between this algorithm and continued fractions is well studied, see
Section 4.5.3 \cite{Knu1998:xiv+762}.

For $1 \leq i \leq K$, we choose $r_i$ and $\ep_i$ so that
\begin{align}
  r_{i-1} = q_{i+1} r_i - \ep_{i+1} r_{i+1} & \label{eqn:ri-recursion}\\
  r_{i+1} \leq r_i/2 & \qquad \quad \text{and} \label{eqn:ri-half} \\
  \text{if $r_{i+1} = r_i/2$, then $\ep_{i+1} = -1$.} \label{eqn:ri-epi} &
\end{align}
For each $i$, we can write $\ep_i r_i$ in terms of $N$ and $s$:
\[ m_i' N + \bar \ep_i r_i = m_i s, \]
for some $m_i, m_i'>0$, where $\bar \ep_i = \ep_1 \dotsb \ep_i$.
One should interpret this as ``using $m_i$ $s$-steps (positive ones only), one can reach $\bar \ep_i r_i$ by going around the circle $m_i'$ times.''

Using the expansion in (\ref{eqn:euclidalg}), we get
\[ C N + \bar \ep_{i+1} r_{i+1} = (q_{i+1} m_i - \ep_i m_{i-1})s. \]
Immediately, we see that $m_i$ satisfy the recurrence relation
\begin{align*}
  m_{i+1}   &= q_{i+1} m_i - \ep_i m_{i-1}
\end{align*}
with base cases $m_{-1} = 0$ and $m_0 = 1$.
The next few elements in the series are
\begin{align*}
m_1 &= q_1 &
m_2 &= q_1 q_2 - \ep_1 &
m_3 &= q_1 q_2 q_3 - \ep_1 q_3 - \ep_2 q_1
\end{align*}

As we will show, the sequence of $m_i$'s for $-1 \leq m_i \leq K$ is non-negative and strictly increasing.
In addition, $m_i$ is the smallest positive integer, $\ell$ such that $\ell s = \bar \ep_i r_i \mod N$.

\begin{thm}
  \[
    \dia \asymp \min_{-1 \leq i \leq K} \{ \max \{ r_i^\aI, m_{i+1}^\aII \} \}.
  \]
  In other words, there exist constants $c_1, c_2 > 0$ such that
  \[ c_1  \min_i \{ \max \{ r_i^\aI, m_{i+1}^\aII \} \}  \leq \dia \leq c_2  \min_i \{ \max \{ r_i^\aI, m_{i+1}^\aII \} \}. \]
  In particular we can set $c_1 = 1 / 2^{5(\aI + \aII)}$ and $c_2 = 1$.
  \label{thm:N1s-diameter}
\end{thm}

First, we present some simple corollaries.

\begin{cor}
\begin{enumerate}
  \item Let $N = s t$, where $s, t > 0$, $G = \Cn$, $S = (1,s)$, and $\sp = (\aI, \aII)$. Then,
    \[ \dia \asymp \min \left\{ N^\aI, \max \{ s^\aI, t^\aII \} \right\}. \]
  \item   Suppose $s^\aI \leq (N/s)^\aII$.
    Let $N = s t$, where $s, t > 0$, $G = \Cn$, $S = (1,s)$, and $\sp = (\aI, \aII)$.
    Then,
    \[ \dia \asymp \min\{ N^\aI, (N/s)^\aII \}. \]
  \item Let $N = s t_1 + s_2$, where $0 \leq s_2 \leq s/2$ and $s = s_2 t_2$, $G = \Cn$, $S = (1,s)$, and $\sp = (\aI, \aII)$.
    Then,
    \[ \dia \asymp \min\{ N^\aI, \max\{s^\aI, t_1^\aII\}, \max\{(s_2^\aI, (t_1t_2)^\aII\} \}. \]
  \item Let
    \begin{align*}
      N   &= s   t_1 + s_2 \\
      s   &= s_2 t_2 + s_3 \\
      s_2 &= s_3 t_3,
    \end{align*}
    where $0 < s$, $0 < s_2 \leq s/2$, and $0 < s_3 \leq s_2/2$,
    $G = \Cn$, $S = (1,s)$, and $\sp = (\aI, \aII)$.
    Then
    \[ \dia \asymp \min\{ N^\aI, \max\{s^\aI, t_1^\aII\}, \max\{(s_2^\aI, (t_1t_2)^\aII\},
      \max\{ s_3^\aI, (t_1 t_2 t_3)^\aII\}
      \}. \]
\end{enumerate}
\label{ex:diam-alg}
\end{cor}

\begin{proof}[Proof of the upper bound in Theorem \ref{thm:N1s-diameter}]
It suffices to show that for all $x \in \Cn$,
\[ \cost{x} \leq \min_i \{ \max \{ r_i^\aI, m_{i+1}^\aII \} \}. \]
Fix $i$. For all $x \in \Cn$, $|ks - x| < r_{i}$ for some $0 \leq k \leq m_{i+1}$.
Then $x = k s + r$ for some $|r| \leq r_i$ and $|k| \leq m_{i+1}$. Therefore,
$\cost{x} \leq  \max \{ r_i^\aI, m_{i+1}^\aII \} ,$
for all $i$. Taking the minimum over all $i$, we achieve the desired result.
\end{proof}

Our proof of the lower bound is much more involved, and will use the following definition and proposition
\begin{defn}
  Fix positive integers $N$ and $s$ with $0 < s \leq N/2$, define
  \[ [x]_s = \argmin_{\ell \in \Z} \{ |\ell| : \ell s \equiv x \mod N \},\]
  and if there are two options, choose the positive one.
\end{defn}

\begin{figure}
\begin{tikzpicture} [scale=1.8,
		     every node/.style={circle,inner sep=0pt,minimum size=4pt,outer sep=5pt},]

\draw (0,0) circle [radius=2cm];

\foreach \xi in {0,1,...,23}
{
  \def\angle{90-\xi*360/47}
  \draw[thin,line width=1pt] (\angle:1.9cm) -- (\angle:2.1cm);
  \node[font=\Tiny] at (\angle:1.75cm) {\textsf{\xi}};
}

\foreach \xi in {-1,-2,...,-23}
{
  \def\angle{90-\xi*360/47}
  \draw[thin,line width=1pt] (\angle:1.9cm) -- (\angle:2.1cm);
  \node[font=\Tiny] at (\angle:1.75cm) {\textsf{\xi}};
}

\node[font=\tiny] (s0) at (90:2.3cm) {\textsf{0}};

\def\s{360*17/47}

\foreach \xi in {1,-1,2,-2}
{
  \node[fill,yellow] (f\xi) at (90-\xi*\s:2cm) {};
  \node[font=\tiny] (s\xi) at (90-\xi*\s:2.3cm) {\textsf{\xi}};
  \ifthenelse{\xi=1}{\node[red,rectangle,draw,rotate fit=90-\xi*\s,fit=(f\xi)(s\xi),label={0:$s$}] {};}{}
}

\foreach \xi in {3,4,...,10,-3,-4,...,-10}
{
  \node[fill,red] (f\xi) at (90-\s*\xi:2cm) {};
  \node[font=\tiny] (s\xi) at (90-\s*\xi:2.3cm) {\textsf{\xi}};
  \ifthenelse{\xi=3}{\node[red,rectangle,draw,rotate fit=90-\s*\xi,fit=(f\xi)(s\xi),label={0:$r_1$}] {};}{}
}

\foreach \xi in {11,12,...,23,-11,-12,...,-23}
{
  \node[fill,gray] (f\xi) at (90-\s*\xi:2cm) {};
  \node[font=\tiny] (s\xi) at (90-\s*\xi:2.3cm) {\textsf{\xi}};
  \ifthenelse{\xi=11}{\node[red,rectangle,draw,rotate fit=90-\s*\xi,fit=(f\xi)(s\xi),label={0:$r_2$}] {};}{}
}

\end{tikzpicture}
\caption{
  We visualize $\Z/47 \Z$, whose elements are labeled inside the circles. With $s = 17$, we label $[\cdot]_s$ for each element on the outside of the circle. Following the positive $s$-steps, we can see that $r_1$ is the first time that the path has visited $(-s/2,s/2)$, at which time the path turn from yellow to red. Similarly, at the 11$^{th}$ $s$-step, we visit $-1$ and it is the first time we visit the interval $(-r_1/2,r_1/2)$.
 }
\end{figure}

\begin{prop}
  For all $0 \leq i \leq K$, $| [r_i]_s |= m_i.$
  \label{prop:sprop}
\end{prop}

First we give some properties of $[\cdot]_s$ in the following lemma:

\begin{lem}
  \begin{enumerate}
    \item For any $x \in \Z/N\Z$, represented as $0 < x \leq N/2$, if $[x]_s = -[x]_s$, then $x$ divides $N$.
    \item For all $i < K$, $[-r_i]_s = - [r_i]_s$.
    \item Let $x, y \in \Z / N \Z$, where $[x]_s$ and $[y]_s$ are positive.
      If $[x + y]_s \geq \min ([x]_s, [y]_s)$, then  \[ [x + y]_s = [x]_s + [y]_s. \]
    \item Let $x, y \in \Z / N \Z$, where $[x]_s$ and $[y]_s$ are positive.
      If $[x]_s > [y]_s > 0$, then $[x - y]_s = [x]_s - [y]_s$.
  \end{enumerate}
  \label{lem:ssteps}
\end{lem}

\begin{proof}
  \begin{enumerate}
    \item By definition, $[-x]_s$ is either $-[x]_s$ or $[x]_s$. If the former is true, then we are done.
      If the latter, then let $\ell = [x]_s = [-x]_s$, and then we have
      \begin{align*}
        \ell s &\equiv -x\mod N \\
        \ell s &\equiv x\mod N .
      \end{align*}
      Thus, $2 x \equiv 0 \mod N$.
    \item Follows directly from (1).
    \item Let $a = [x]_s$ and $b = [y]_s$. Without loss of generality, we can assume that $a > b$.
      We know that $[x + y]_s \leq a+b$ by definition.
      Thus suppose that $0 < [x + y]_s < a+b$ would imply that $0 < [x + y]_s - b < a$, which contradicts the assumption that $[x]_s = a$.
    \item
      We know that $[x - y]_s \leq a-$ by definition.
      If $[x - y]_s < a-b$, then it would be true that $0 < [x - y]_s + b < a$, which contradicts the assumption that $[x]_s = a$.

  \end{enumerate}
\end{proof}

\begin{proof}[Proof of Proposition \ref{prop:sprop}]
  We prove by induction with the following induction hypotheses:
  \begin{enumerate}[label=(\alph*)]
    \item $[\bar \ep_i r_i]_s = m_i$ (note in particular, this means that $[\bar \ep_i r_i]_s$ is positive), where $\bar \ep_i = \ep_1 \cdots \ep_i$.
    \item For all $x \in S_i = (- 2r_i - r_{i-1}, 2r_i + r_{i-1}) \setminus \{\pm r_{i-1}, 0 \}$, $|[x]_s| > m_i$.
  \end{enumerate}

  The base cases are for $i = 0$ and $1$, i.e. $r_i$ being $s$ and $r_1$, which are trivial for both hypotheses.

  For the induction step, we first consider when $i < K-1$, and case-split based on the signs of $- r_{i-1}, -r_i, r_i,$ and $r_{i-1}$:

  If the signs are $(-,-,+,+)$, then we know that $\bar \ep_{i-1} = \ep_i = 1$.
  Then by Lemma \ref{lem:ssteps} (3)  and the induction hypothesis, we can deduce that $[r_i]_s = m_i$, $[2r_i]_s = 2m_i$, $\dots$, and $[q_{i+1}r_i]_s = q_{i+1} m_i$.
      Thus,
      \begin{align*}
        [\ep_{i+1} r_{i+1}]_s = [\bar \ep_{i+1} r_{i+1}]_s &= [q_{i+1}r_i - r_{i-1}]_s \\
        & = [q_{i+1} r_i]_s - [r_{i-1}]_s \qquad \text{(Lemma \ref{lem:ssteps} (4))} \\
        & = q_{i+1} m_i - m_{i-1} = m_{i+1}.
       \end{align*}

      For induction hypothesis (b), note that it suffices to show this for points with positive $[\cdot]_s$, by Lemma \ref{lem:ssteps} (1).
      Starting from $-r_{i-1}$, consider the path along positive $s$-steps.
      By the induction hypothesis, the next visit to the set $S_i$ is at $r_i - r_{i-1}$, and the next at $2 r_i - r_{i-1}$, and so on.
      Therefore, the first visit to $S_{i+1}$ is at the point $q_{i+1} r_i - r_{i-1}$, by the definition of $q_{i+1}$.
      This point is $\ep_{i+1} r_{i+1}$. Therefore, for all other points in $S_i$, specifically the ones in $S_{i-1}$, have $[\cdot]_s$ greater than $m_{i+1}$, if it is realized by a positive $s$-path.

    If the signs are $(-,+,-,+)$, then we know that $\bar \ep_{i-1} = 1$ and  $\ep_i = -1$.
      As in the previous case, by Lemma \ref{lem:ssteps} (2) and the induction hypothesis, $[q_{i+1}r_i]_s = - q_{i+1} m_i$, and
      \[ [\bar \ep_{i+1} r_{i+1}]_s = - [\ep_{i+1} r_{i+1}]_s = - [q_{i+1}r_i - r_{i-1}]_s = q_{i+1} m_i + m_{i-1} = m_{i+1}.\]

      We follow the format of the previous case, instead starting at $r_{i-1}$, considering again positive $s$-steps.
      In this case, the next visits to $S_i$ are $r_{i-1} - r_i$, $r_{i-1} - 2 r_i$, $\dotsc$, and $r_{i-1} - q_{i+1} r_i$, which is $\bar \ep_{i+1} r_{i+1}$.
    The last two cases are the same as the cases above with the signs switched.

    Remaining are the cases when $i = K$. The possible sign combinations for $- r_{K-1}, -r_K, r_K,$ and $r_{K-1}$
    are $(-, +, +, +)$ and $(+, +, +, -)$. Then the argument proceeds exactly the same as above.
    \qedhere
\end{proof}

\begin{lem}
  \begin{enumerate}
    \item \label{item:ratio}
    Let $r > 0$ and $x \geq n$, where $n$ is a positive integer. Then
    \[\floor{rx} \geq \frac{\floor{rn} - 1}{\floor{rn}} rx
    \qquad \text{and} \qquad
    \ceil{rx} \leq \frac{\ceil{rn}}{\ceil{rn}-1} rx.
    \]
    Note that the first inequality is only of interest of when $rn > 1$ and the second inequality when $rn > 2$.
    \item \label{item:qi}
      For all $1 \leq i \leq K+1$, $q_i \geq 2$.
    \item \label{item:incr} The $m_i$'s are strictly increasing.
    \item \label{item:qm} Let $2 \leq i \leq K$. Then \[ q_{i+1} m_i \geq \frac{m_{i+1}}{4}. \]
    \item \label{item:qr} Let $i \geq 1$. Then \[ q_{i+1} r_i \geq \frac{3 r_{i-1}}{4}. \]
  \end{enumerate}
  \label{lem:catchall}
\end{lem}

\begin{proof}
  \begin{enumerate}
   \item The proof is simple and we omit it here.
  \item Fix $1 \leq i \leq K+1$.
  \[ r_{i-2} + \ep_i r_i = q_i r_{i-1}. \]
  The algorithm requires that for each $i\geq0$, $r_{i+1} \leq r_i /2$.
  For each $i \geq 1$ we know that $r_{i-2} \geq 2 r_{i-1}$ and $r_{i-1} \leq r_{i} /2$, so
  \[ q_{i} r_{i-1} = r_{i-2} + \ep_{i} r_{i} \geq 2 r_{i-1} - r_{i-1} /2 = (3/2) r_{i-1} \]
  So $q_i \geq 3/2$, and since the $q_i$'s must be positive integers, $q_i \geq 2$.

  \item  From our inductive definition we have for the base case
    \[ m_2 = q_1 q_2 - \ep_1 \geq q_1 (2 - 1/2) \geq (3/2) q_1 > m_1 . \]
  and for the inductive case,
  \begin{align*}
    m_{i+1} &= q_{i+1} m_i - \ep_i m_{i-1} \\
    &\geq 2 m_i - m_{i-1} \qquad \text{(Lemma \ref{lem:catchall} (\ref{item:qi}))} \\
    &> m_i \qquad \text{(induction hypothesis)}.
  \end{align*}

  \item If $i = 1$, we have
  \[ q_{2} m_1 = q_2 q_1 \geq \frac34 (q_2 q_1 - 1) \geq \frac34 (q_2 q_1 - \ep_1) = \frac 34 m_2. \]
  If $i > 1$, then $q_{i+1} m_i = m_{i+1} - \ep_i m_{i-1}$.
  If $m_{i+1}/m_{i-1} \leq 2$, then since $m_i$'s are increasing by Lemma \ref{lem:catchall} (\ref{item:incr}), we know $m_{i+1} \leq 2 m_i$, as well.
  So, \[ q_{i+1} m_{i} \geq q_{i+1} \frac{m_{i+1}}{2} \geq m_{i+1}, \]
  with the last inequality following from $q_{i+1} \geq 2$, Lemma \ref{lem:catchall} (\ref{item:qi}).
  If $m_{i+1}/m_{i-1} \geq 2$,
  \[q_{i+1} m_i \geq m_{i+1} - m_{i-1} \geq \frac{1}{2} m_{i+1}.\]

    \item $r_{i-1} = q_{i+1} r_i - \ep_i r_{i+1}
    \leq q_{i+1} r_i + r_{i+1}
    \leq q_{i+1} r_i + \frac{r_{i-1}}{4}$ \qedhere
  \end{enumerate}
\end{proof}

\begin{prop}
Let $0 \leq i \leq K$. Define
\begin{equation}
  x_i =
  \begin{cases}
    \floor{q_{i+1}/2} r_i &\text{ if } q_{i+1} \geq 8 \\
    r_i                   &\text{ if } q_{i+1} < 8
  \end{cases}.
\end{equation}
Then,
  \[ \cost{x_i} \geq \frac{1}{2^{5 (\aI + \aII)}} \min \{ r_{i-1}^\aI, m_{i+1}^\aII \}.\]
\label{lem:lowerbd-most}
\end{prop}

\begin{proof}
  First we consider the case when $q_{i+1} \geq 8$, and thus $x_i = \floor{q_{i+1}/2} r_i$.
  Let $n_i = \floor{q_{i+1}/4}$.
  Consider the interval $[-r_{i-1}, r_{i-1}]$ with the points reachable using at most $n_i m_i$ large steps,
  and the two colors signify the large steps that were used with generators of the opposite sign.
  The particular picture uses $q_{i+1} = 10$:
  
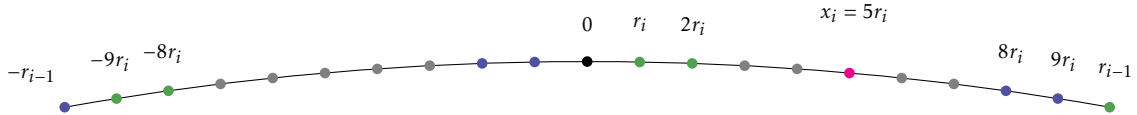
\begin{center}
  \begin{tikzpicture}[every node/.style={circle,inner sep=0pt,minimum size=4pt,outer sep=6pt,font=\scriptsize},
  fsnode/.style={fill,verte,font=},
  ssnode/.style={fill,bleu},
  bsnode/.style={fill,gray}
]
   \pgfmathsetmacro{\r}{40}

   \centerarc[](0,0)(80:100:\r);
   \node[fsnode] at (80:\r) {};
   \node[fill,black] at (90:\r) {};
   \node at (90:{\r+.5}) {$0$};
   \node at (80:{\r+.5}) {$r_{i-1}$};

   \node[fsnode] at (89:\r) {};
   \node at (89:{\r+.5}) {$r_{i}$};
   \node[fsnode] at (88:\r) {};
   \node at (88:{\r+.5}) {$2 r_i$};

   \node[ssnode] at (91:\r) {};
   \node[ssnode] at (92:\r) {};
   \node[ssnode] at (100:\r) {};
   \node at (100.5:{\r+.5}) {$- r_{i-1}$};

   \node[fsnode] at (99:\r) {};
   \node at (99:{\r+.5}) {$-9r_i$};
   \node[fsnode] at (98:\r) {};
   \node at (98:{\r+.5}) {$-8r_i$};

   \node[ssnode] at (81:\r) {};
   \node at (81:{\r+.5}) {$9r_i$};
   \node[ssnode] at (82:\r) {};
   \node at (82:{\r+.5}) {$8r_i$};

   \node[fill,magenta] at (85:\r) {};
   \node at (85:{\r+0.8}) {$x_i = 5r_i$};

   \node[bsnode] at (83:\r) {};
   \node[bsnode] at (84:\r) {};
   \node[bsnode] at (86:\r) {};
   \node[bsnode] at (87:\r) {};
   \node[bsnode] at (93:\r) {};
   \node[bsnode] at (94:\r) {};
   \node[bsnode] at (95:\r) {};
   \node[bsnode] at (96:\r) {};
   \node[bsnode] at (97:\r) {};
\end{tikzpicture}

  \end{center}
  Consider a path that $w$ that maps to $x_i$ under the standard projection.
  If $\deg_{\s} (w) \geq n_i m_i$ or more large steps,
  \begin{align*}
    n_i m_i &= \floor{\frac{q_{i+1}}{4}} m_i  \\
    & \geq \frac{q_{i+1} m_i}{2^3} \quad \text{(Lemma \ref{lem:catchall} (\ref{item:ratio}))} \\
    & \geq \frac{1}{2^5} m_{i+1} \quad \text{(Lemma \ref{lem:catchall} (\ref{item:qm}))}
  \end{align*}
  If $\deg_{\s} (w) \leq n_i m_i$, then
  \begin{align*}
    \deg_{1} (w) & \geq x_i - n_i r_i = \left( \floor{\frac{q_{i+1}}{2}} - \floor{\frac{q_{i+1}}{4}} \right) r_i \\
     &\geq \left( \frac{3}{4} \frac{q_{i+1}}{2} - \frac{q_{i+1}}{4} \right) r_i \quad \text{(Lemma \ref{lem:catchall} (\ref{item:ratio}))} \\
     &=  \frac{q_{i+1}}{2^3} r_i \\
     &\geq \frac{r_{i-1}}{2^5} \quad \text{(Lemma \ref{lem:catchall} (\ref{item:qr}).)} \qedhere
  \end{align*}

  Now consider consider the case when $q_{i+1} < 8$, and thus $x_i = r_i$.
  If we use fewer than $m_i$ large steps, then the number of small steps required is
  \[
    r_i \geq \frac{q_{i+1}r_i}{2^3} \geq \frac{r_{i-1}}{2^5}.
  \]
  If we can use $m_i$ large steps, then we can reach $r_i$. But
  \[
  m_i \geq \frac{q_{i+1}m_i}{2^3} \geq \frac{m_{i+1}}{2^5}. \qedhere
  \]
\end{proof}

\begin{proof}[Proof of the lower bound of Theorem \ref{thm:N1s-diameter}]
For the lower bound, let
\[
  L = \argmin_i \{ \max \{ r_{i}^\aI, m_{i+1}^\aII \} \}.
\]
Suppose $L = -1$. Then $N^\aI < q_1^\aII$ and
\begin{align*}
  \dia \geq \cost{x_0} \geq \frac{1}{2^{5(\aI+\aII)}} N^\aI.
\end{align*}

Now we can assume that $L \geq 0$,
\begin{description}
  \item[Case 1 ($r_L^\aI \leq m_{L+1}^\aII$) ]
    Then $m_{L+1}^\aII = \max \{ r_{L}^\aI, m_{L+1}^\aII \} \leq \max \{ r_{L-1}^\aI, m_{L}^\aII \} = r_{L-1}^\aI$,
    since the $m_i$'s are increasing.
    By substituting $i = L$ into Lemma \ref{lem:lowerbd-most}, we have that
    \[ \dia \geq \frac{1}{2^{5(\aI + \aII)}} \min \{ r_{L-1}^\aI, m_{L+1}^\aII \} \geq \frac{m_{L+1}^\aII}{2^{5(\aI + \aII)}}.\]
  \item[Case 2 ($r_L^\aI \geq m_{L+1}^\aII$) ]
    Then $r_L^\aI = \max \{ r_{L}^\aI, m_{L+1}^\aII \} \leq \max \{ r_{L+1}^\aI, m_{L+2}^\aII \} = m_{L+2}^\aII$,
    since the $r_i$'s are decreasing.
    By substituting $i = L+1$ into Lemma \ref{lem:lowerbd-most}, we have that
    \[ \dia \geq \frac{1}{2^{5(\aI + \aII)}} \min \{ r_{L}^\aI, m_{L+2}^\aII \} \geq \frac{r_{L}^\aI}{2^{5(\aI + \aII)}}. \qedhere\]
\end{description}
\end{proof}

\bibliographystyle{alpha}
\bibliography{nilgrp}

\end{document}